\documentclass[10pt]{article}
\usepackage{tikz}
\usetikzlibrary{arrows.meta, positioning}

\usepackage[margin=1in]{geometry}
\usepackage{amsmath,amsfonts,amssymb,mathrsfs,nccmath,amscd}
\usepackage{enumerate,float,graphicx}

\usepackage{graphicx,wrapfig}
\usepackage{subcaption}
\usepackage[thinlines,thiklines]{easybmat}
\usepackage{hyperref}
\usepackage{capt-of}
\usepackage{float}
\usepackage{stackrel}
\usepackage{tikz-cd}
\usepackage{nicematrix,tikz}
\usetikzlibrary{tikzmark}
\usetikzlibrary{fit}

\usepackage{setspace}
\usepackage{tabu}
\usepackage{booktabs}

\usepackage{fancyhdr}
\newtheorem{theorem}{Theorem}[section]
\newtheorem{lemma}[theorem]{Lemma}
\newtheorem{remark}[theorem]{Remark}
\newtheorem{proposition}[theorem]{Proposition}
\newtheorem{corollary}[theorem]{Corollary}
\newtheorem{definition}[theorem]{Definition}
\newtheorem{example}[theorem]{Example}
\newenvironment{proof}{\begin{trivlist} \item[]{\em Proof.}}{\end{trivlist}}
\newcount\refno
\newcommand\be{\begin{equation}}
\newcommand\ee{\end{equation}}
\newcommand\bn{\begin{eqnarray}}
\newcommand\en{\end{eqnarray}}
\newcommand\bns{\begin{eqnarray*}}
\newcommand\ens{\end{eqnarray*}}
\newcommand\bd{\begin{definition}}
\newcommand\ed{\end{definition}}
\newcommand\br{\begin{remark}}
\newcommand\er{\end{remark}}
\newcommand\bt{\begin{theorem}}
\newcommand\et{\end{theorem}}
\newcommand\bp{\begin{proposition}}
\newcommand\ep{\end{proposition}}
\newcommand\bc{\begin{corollary}}
\newcommand\ec{\end{corollary}}
\newcommand\bl{\begin{lemma}}
\newcommand\el{\end{lemma}}
\newcommand\pf{\begin{proof}}
\newcommand\qed{\end{proof}\eop}

\newcommand\bR{{\mathbb R}}
\newcommand\bN{{\mathbb N}}

\newcommand\cD{{\cal D}}
\newcommand\cE{{\cal E}}

\newcommand\cR{{\cal R}}

\newcommand{\N}{\mbox{$\mathbb{N}$}}
\newcommand{\NS}{\mbox{\scriptsize${\mathbb{N}}$}}

\newcommand{\F}{\mbox{$\mathcal{F}$}}

\renewcommand{\L}{\mbox{$\mathcal{L}$}}

\def\eop{\hfill\rule{2.0mm}{2.0mm}}

\makeindex
\begin{document} 

\allowdisplaybreaks
\title{Riordan array representation of recursive polynomial sequences, orthogonal polynomial sequences, and $d$-orthogonal polynomial sequences}

\author{Tian-Xiao He}

\date{}

\maketitle
\setcounter{page}{1}
\pagestyle{myheadings}
\markboth{From T.-X. He}
{Riordan array representation of recursive polynomial sequences}

\begin{abstract}
We study when a polynomial sequence $\{p_n(x)\}$ satisfying a linear homogeneous recurrence 
with polynomial coefficients admits an ordinary Riordan array as its coefficient matrix. For second-order recurrences $p_{n+2}(x)=a(x)p_{n+1}(x)+b(x)p_n(x)$, we give a complete characterization: the coefficient matrix is a Riordan array $(g,f)$ if and only if $a(x),b(x)$ are affine in $x$, the initial term $p_0(x)$ is constant, and $p_1(x)-a(x)p_0(x)$ is constant, with $g,f$ given explicitly in terms of these constants. We extend this to a necessary condition and a full factorization criterion for recurrences of arbitrary order $\ell \geq 3$. As an application we recover, unify, and correct several known identities for the generalized Gegenbauer--Humbert polynomials $P_n^{\lambda,C,y}(x)$ and their special cases (Chebyshev, Legendre, Pell, Fibonacci, Fermat, and Dickson polynomials), including explicit Riordan array representations, row-sum and rising-diagonal-sum identities, and derivative evaluation formulas expressed via the fundamental theorem of Riordan arrays.

We then develop a production-matrix framework for arbitrary orthogonal polynomial sequences (OPS), not restricted to the Riordan-array-type case: for any lower-triangular invertible coefficient matrix $A$ of an OPS,  the production matrix $P_{A^{-1}}=ASA^{-1}$ is always the tridiagonal Jacobi matrix encoding the three-term recurrence, and the first column of $A^{-1}$ always gives the moment sequence -- even when, as for the Legendre polynomials, the recurrence coefficients are non-constant and no Riordan array representation of the coefficient matrix exists. We illustrate this with Legendre and Chebyshev polynomials as, respectively, non-Riordan-type and Riordan-type examples, and give an explicit closed-form coefficient matrix and its inverse for the generalized Gegenbauer--Humbert OPS. Finally, we extend the framework to $d$-orthogonal polynomial sequences, showing that the coefficient matrix of a $d$-orthogonal sequence is always invertible, that its associated production matrix is $(d+2)$-banded (lower Hessenberg) and encodes the corresponding $(d+2)$-term recurrence, and that the first $d$ columns of the inverse matrix recover the moment sequences of the defining vector functional -- identifying $d$-orthogonality with $(d+1)$-Hessenberg generalized Riordan arrays and extending the classical $d=1$ tridiagonal correspondence.


\end{abstract}

\section{Introduction}
Riordan arrays are infinite, lower triangular matrices defined by the generating function of their columns. They form a group, denoted by $\mathcal{R}$ and called {\em the Riordan group} (see Shapiro, Getu, W. J. Woan and L. Woodson \cite{SGWW}). 

A Riordan array $(g(t),f(t))$ is a pair of formal power series $g(t) = \sum_{n\geq 0}g_nt^n$ and $f(t) = \sum_{n\geq 1}f_nt^n$, with $g_0\not= 0$ and $f_1\not= 0$. It defines an infinite lower triangular array $[d_{n,k}]_{n,k\geq 0}$ according to the rule $d_{n,k} = [t^n ]g(t)f(t)^k$. The set of all Riordan arrays forms a group under matrix multiplication

\[
(g(t), f (t))(h(t), l(t)) = (g(t)h(f (t)), l(f (t))).
\]

More formally, let us consider the set of all formal power series (f.p.s.) in $t$, $\F = {\mathbb K}[\![$$t$$]\!]$, with a field ${\mathbb K}$ of characteristic $0$ (e.g., ${\mathbb Q}$, ${\mathbb R}$, ${\mathbb C}$, etc.). The \emph{order} of $f(t)  \in \F$, $f(t) =\sum_{k=0}^\infty f_kt^k$ ($f_k\in {\bR}$), is the minimal number $r\in\N$ such that $f_r \neq 0$. Denote by $\F_r$ the set of formal power series of order $r$. Let $g(t) \in \F_0$ and $f(t) \in \F_1$. Then, the pair $(g(t) ,\,f(t) )$ defines the {\em (proper) Riordan array} $D=(d_{n,k})_{n,k\in \NS}=(g(t), f(t))$ having
  
\begin{equation}\label{Radef}
d_{n,k} = [t^n]g(t) f(t) ^k
\end{equation}
or, in other words, having $g(t) f(t)^k$ as the generating function whose coefficients make-up the entries of column $k$. 

From the {\it fundamental theorem of Riordan arrays} (see \cite{SGWW}), it is immediate to show that the usual row-by-column product of two Riordan arrays is also a Riordan array:
\begin{equation}\label{Proddef}
    (g_1(t) ,\,f_1(t) )  (g_2(t) ,\,f_2(t) ) = (g_1(t) g_2(f_1(t) ),\,f_2(f_1(t) )).
\end{equation}
The Riordan array $I = (1,\,t)$ acts as an identity for this product. Thus, the set of all Riordan arrays forms the Riordan group $\mathcal{R}$.

Several subgroups of $\mathcal{R}$ are important and have been considered in the literature:
\begin{itemize} \item the set $\mathcal{A}$ of {\em Appell arrays} is the collection of all Riordan arrays $R = (g(t) ,\,t )$ in ${\cR}$; 
\item the set $\mathcal{L}$ of {\em Lagrange arrays} is the collection of all Riordan arrays $R = (1 ,\,f(t) )$ in ${\cR}$;
\item the set $\mathcal{B}$ of \emph{Bell} or {\em renewal arrays} is the collection of all Riordan arrays $R = (g(t) ,\,t g(t))$ in ${\cR}$;
\item the set $\mathcal{H}$ of \emph{hitting-time arrays} is the collection of all Riordan arrays $R = (tf'(t)/f(t) ,\,f(t))$ in ${\cR}$;
\item the set $\mathcal{D}$ of the Riordan arrays $R = (f'(t) ,\, f(t) )$ in ${\cR}$ is called the {\it derivative subgroup}.
\item the set $\mathcal{E}$ of the Riordan arrays $R=((f(t)/t)^{r}f'(t)^{s}, f(t))$ in ${\cR}$ with real or complex $r$ and $s$ is called {\it Lu\'zon-Merlini-Mor\'on-Sprugnoli (LMMS) subgroup}, denoted by ${\cE}[r,s]$, which includes ${\cD}={\cE}[0,1]$ as its special case (see \cite{LMMS2}). 
\end{itemize}

From \cite{Rog}, an infinite lower triangular array $[d_{n,k}]_{n,k\in{\bN}}=(g(t), f(t))$ is a Riordan array if and only if an {\it $A$-sequence} $A=(a_0\not= 0, a_1, a_2,\ldots)$ exists such that for every $n,k\in{\bN}$ there holds 
\be\label{eq:1.1}
d_{n+1,k+1} =a_0 d_{n,k}+a_1d_{n,k+1}+\cdots +a_nd_{n,n},
\ee 
which is shown in \cite{HS09} to be equivalent to 
\be\label{eq:1.2}
f(t)=tA(f(t)).
\ee
Here, $A(t)$ is the generating function of the $A$-sequence. In \cite{MRSV} it is also shown that a unique {\it $Z$-sequence} $Z=(z_0, z_1,z_2,\ldots)$ exists such that every element in column $0$ can be expressed as the linear combination 
\be\label{eq:1.3}
d_{n+1,0}=z_0 d_{n,0}+z_1d_{n,1}+\cdots +z_n d_{n,n},
\ee
or equivalently (see \cite{HS09}),
\be\label{eq:1.4}
g(t)=\frac{d_{0,0}}{1-tZ(f(t))}.
\ee

Denote the {\it upper shift matrix} by $U$, i.e., 
\[
U=(\delta_{i+1,j})_{i,j\geq 0}=\left [ \begin{array}{llllll} 0& 1& 0& 0& 0& \cdots\\
0&0 &1& 0& 0&\cdots\\
0 &0& 0&1& 0& \cdots\\
0&0& 0& 0&1&\cdots\\
\vdots &\vdots& \vdots& \vdots&\vdots&\ddots\end{array}\right]
\]
and 
\be\label{eq:1.4-2}
P=\left [ \begin{array}{llllll} z_0& a_0& 0& 0& 0& \cdots\\
z_1& a_1 &a_0& 0& 0&\cdots\\
z_2 &a_2& a_1& a_0& 0& \cdots\\
z_3& a_3& a_2& a_1&a_0&\cdots\\
\vdots &\vdots& \vdots& \vdots&\vdots&\ddots\end{array}\right]=\left( Z(t), A(t), tA(t), t^{2}A(t),\ldots\right), 
\ee
where the rightmost expression is the representation of $P$ by using its column generating functions. Here, $P$ is called the {\it production matrix} or {\it $P$-matrix characterization} or simply {\it $P$ matrix} (see Deutsch, Ferrari, and Rinaldi 
\cite{DFR05, DFR}). From \cite{DFR05, DFR} or Proposition 2.7 of \cite{He15}, the $P$-matrix of Riordan array $R$ satisfies 

\be\label{1.5}
P =R^{-1}UR=R^{-1}\overline{R},
\ee
where $\overline{R}$ is the truncated Riordan array $R$ with the first row omitted. In Section 3, the notation of production matrix will be extended to any invertible lower triangular matrices. 

The following theorem gives the ordinary generating function of a polynomial sequence satisfying a homogeneous linear recurrence of order $\ell$.

\begin{theorem}\label{thm:lth-order-generating-function}
Let $\{p_n(x)\}_{n\ge0}$ be a sequence of polynomials satisfying
\begin{equation}
p_{n+\ell}(x)=\sum_{j=1}^{\ell}a_j(x)p_{n+\ell-j}(x), \qquad n\ge0,
\label{eq:lth-order-recurrence}
\end{equation}
where
\[
a_1(x),a_2(x),\ldots,a_\ell(x)
\]
are fixed polynomials in $x$.

Assume that the initial polynomials
\[
p_0(x),p_1(x),\ldots,p_{\ell-1}(x)
\]
are prescribed.

Define the ordinary generating function
\begin{equation}
P(t,x)=\sum_{n=0}^{\infty}p_n(x)t^n.
\label{eq:ordinary-generating-function}
\end{equation}

Then
\begin{equation}
P(t,x)=\frac{N(t,x)}{D(t,x)},
\label{eq:lth-generating-function}
\end{equation}
where
\begin{equation}
D(t,x)=1-a_1(x)t-a_2(x)t^2-\cdots-a_\ell(x)t^\ell,
\label{eq:lth-denominator}
\end{equation}
and
\begin{equation}
N(t,x)=\sum_{m=0}^{\ell-1}\left(p_m(x)-\sum_{j=1}^{m}a_j(x)p_{m-j}(x)\right)t^m.
\label{eq:lth-numerator}
\end{equation}

Equivalently,
\begin{align}
N(t,x)={}&p_0(x)+\bigl(p_1(x)-a_1(x)p_0(x)\bigr)t
\nonumber\\
&+\bigl(p_2(x)-a_1(x)p_1(x)-a_2(x)p_0(x)\bigr)t^2
\nonumber\\
&+\cdots\nonumber\\
&+\left(p_{\ell-1}(x)-\sum_{j=1}^{\ell-1}a_j(x)p_{\ell-1-j}(x)\right)t^{\ell-1}.
\label{eq:lth-numerator-expanded}
\end{align}
\end{theorem}

\begin{proof}
From the definition of the generating function,

\[
P(t,x)=\sum_{n=0}^{\infty}p_n(x)t^n.
\]

Hence

\[
P(t,x)-\sum_{m=0}^{\ell-1}p_m(x)t^m=\sum_{n=\ell}^{\infty}p_n(x)t^n.
\]

Using the recurrence relation \eqref{eq:lth-order-recurrence},

\[
p_n(x)=\sum_{j=1}^{\ell}a_j(x)p_{n-j}(x),\qquad n\ge\ell,
\]
we obtain

\begin{align}
\sum_{n=\ell}^{\infty}p_n(x)t^n&=\sum_{j=1}^{\ell}a_j(x)\sum_{n=\ell}^{\infty}p_{n-j}(x)t^n \nonumber\\
&=\sum_{j=1}^{\ell}a_j(x)t^j\sum_{m=\ell-j}^{\infty}p_m(x)t^m.
\label{eq:shifted-sums}
\end{align}

Since

\[
\sum_{m=\ell-j}^{\infty}p_m(x)t^m=P(t,x)-\sum_{m=0}^{\ell-j-1}p_m(x)t^m,
\]
equation \eqref{eq:shifted-sums} becomes

\[
P(t,x)-\sum_{m=0}^{\ell-1}p_m(x)t^m=\sum_{j=1}^{\ell}a_j(x)t^j\left(P(t,x)-\sum_{m=0}^{\ell-j-1}p_m(x)t^m\right).
\]

Collecting the terms containing $P(t,x)$ yields

\begin{align}
P(t,x)\left(1-\sum_{j=1}^{\ell}a_j(x)t^j\right)&=\sum_{m=0}^{\ell-1}p_m(x)t^m\nonumber\\
&\quad-\sum_{j=1}^{\ell}a_j(x)t^j\sum_{m=0}^{\ell-j-1}p_m(x)t^m.
\label{eq:collecting-terms}
\end{align}

It remains to determine the coefficient of each power $t^m$ on the
right-hand side. For $0\le m\le \ell-1$, the coefficient of $t^m$ is

\[
p_m(x)-\sum_{j=1}^{m}a_j(x)p_{m-j}(x).
\]

Therefore the right-hand side of \eqref{eq:collecting-terms} equals

\[
\sum_{m=0}^{\ell-1}\left(p_m(x)-\sum_{j=1}^{m}a_j(x)p_{m-j}(x)\right)t^m.
\]
Consequently,

\[
P(t,x)=\frac{\displaystyle\sum_{m=0}^{\ell-1}\left(p_m(x)-\sum_{j=1}^{m}a_j(x)p_{m-j}(x)\right)t^m}
{\displaystyle1-a_1(x)t-a_2(x)t^2-\cdots-a_\ell(x)t^\ell},
\]
which proves \eqref{eq:lth-generating-function}.
\end{proof}

\begin{corollary}[Second-Order Case]
For $\ell=2$,

\[
p_{n+2}(x)=a_1(x)p_{n+1}(x)+a_2(x)p_n(x),
\]
the generating function is

\[
P(t,x)=\frac{p_0(x)+\bigl(p_1(x)-a_1(x)p_0(x)\bigr)t}{1-a_1(x)t-a_2(x)t^2}.
\]
\end{corollary}

\begin{corollary}[Third-Order Case]
For $\ell=3$,

\[
p_{n+3}(x)=a_1(x)p_{n+2}(x)+a_2(x)p_{n+1}(x)+a_3(x)p_n(x),
\]

the generating function is

\[
P(t,x)=\frac{p_0(x)+\bigl(p_1(x)-a_1(x)p_0(x)\bigr)t+\bigl(p_2(x)-a_1(x)p_1(x)-a_2(x)p_0(x)\bigr)t^2}
{1-a_1(x)t-a_2(x)t^2-a_3(x)t^3}.
\]
\end{corollary}

A sequence of the generalized Gegenbauer-Humbert polynomials 
$\{ P^{\lambda, y, C}_n(x)\}_{n\geq 0}$ is defined by the expansion 
(see, for example, \cite{Com}, Dunkl and Xu \cite{DX}, Gould \cite{Gould}, Lidl, Mullen, and Turnwald\cite{LMT}, one of the authors with Hsu and Shiue \cite{HHS08}) 

\begin{definition}\label{def:5}
The generalized Gegenbauer--Humbert polynomial sequence
\[
\{P_n^{\lambda,C,y}(x)\}_{n\ge0}
\]
is defined by the generating function
\begin{equation}\label{eq:10}
(C-2xt+yt^2)^{-\lambda} = \sum_{n=0}^{\infty} P_n^{\lambda,C,y}(x)t^n,
\end{equation}
where
\[
\lambda>0,\qquad C\neq0,\qquad y\in\mathbb{R}.
\]

When $y$ is regarded as a variable rather than a parameter, we shall
write
\[
P_n^{(\lambda,C)}(x,y)
\]
instead of
\[
P_n^{\lambda,C,y}(x),
\]
emphasizing that the resulting polynomial is bivariate in the variables
$x$ and $y$.

\end{definition}

Differentiating both sides of (\ref{eq:10}) with respect to $t$ gives
\begin{equation}\label{eq:12}
2\lambda(x-yt)(C-2xt+yt^2)^{-\lambda-1} = \sum_{n=1}^{\infty} nP_n^{\lambda,C,y}(x)t^{\,n-1}.
\end{equation}

Substituting (\ref{eq:10}) into the left-hand side of
(\ref{eq:12}) and comparing coefficients of $t^n$ yield
\[
C(n+1)P_{n+1}^{\lambda,C,y}(x) = 2x(\lambda+n)P_n^{\lambda,C,y}(x)
-y(2\lambda+n-1)P_{n-1}^{\lambda,C,y}(x).
\]

Replacing $n+1$ by $n$ gives the recurrence
\begin{equation}\label{eq:13}
P_n^{\lambda,C,y}(x) = \frac{2x(\lambda+n-1)}{Cn} P_{n-1}^{\lambda,C,y}(x)
-\frac{y(2\lambda+n-2)}{Cn}P_{n-2}^{\lambda,C,y}(x), \qquad n\ge2,
\end{equation}
with initial values
\[
P_0^{\lambda,C,y}(x)=C^{-\lambda},
\qquad
P_1^{\lambda,C,y}(x)=2\lambda xC^{-\lambda-1}.
\]

In particular, when $\lambda=1$, (\ref{eq:13}) becomes
\begin{equation}\label{eq:14}
P_n^{1,C,y}(x) = \frac{2x}{C} P_{n-1}^{1,C,y}(x)
-\frac{y}{C} P_{n-2}^{1,C,y}(x), \qquad n\ge2,
\end{equation}
where
\[
P_0^{1,C,y}(x)=C^{-1},
\qquad
P_1^{1,C,y}(x)=2xC^{-2}.
\]

\begin{proposition}\label{pro:special}

The generalized Gegenbauer--Humbert polynomials contain many classical
polynomial families as special cases.

\begin{align*}
P_n^{1,1,1}(x) &=U_n(x),
&&\text{Chebyshev polynomials of the second kind},\\
P_n^{\frac12,1,1}(x) &=\psi_n(x),
&&\text{Legendre polynomials},\\
P_n^{1,-1,1}(x) &=P_{n+1}(x),
&&\text{Pell polynomials},\\
P_n^{1,-1,1}\!\left(\frac{x}{2}\right) &=F_{n+1}(x),
&&\text{Fibonacci polynomials},\\
P_n^{1,2,1}\!\left(\frac{x}{2}\right) &=\Phi_{n+1}(x),
&&\text{Fermat polynomials of the first kind},\\
P_n^{1,2a,2}(x) &=D_n(x,a)/2,
&&\text{Dickson polynomials}.
\end{align*}

Here $a$ is a real parameter, and
\[
F_n=F_n(1)
\]
denotes the $n$th Fibonacci number.
\end{proposition}

\begin{theorem}\label{thm:4}(\cite{HS})
Assume that
\[
x\neq\pm\sqrt{Cy}.
\]
Then the generalized Gegenbauer--Humbert polynomials
\[
\{P_n^{1,C,y}(x)\}_{n\ge0}
\]
admit the closed-form representation
\begin{align}
P_n^{1,C,y}(x) &= C^{-n-1}\Bigg[ \frac{x+\sqrt{x^2-Cy}} {2\sqrt{x^2-Cy}}
\left(x+\sqrt{x^2-Cy}\right)^n\nonumber\\
&\qquad\qquad
- \frac{x-\sqrt{x^2-Cy}} {2\sqrt{x^2-Cy}} \left(x-\sqrt{x^2-Cy}\right)^n\Bigg].
\label{eq:11}
\end{align}
\end{theorem}

\begin{proof}
The recurrence (\ref{eq:14}) has constant coefficients. Applying the general Binet formula for second-order linear recurrences, immediately yields (\ref{eq:11}).
\end{proof}

\begin{proposition}\label{pro:5}
In the special case $\lambda=C=1$, then the generalized Gegenbauer--Humbert polynomials reduce to $P^{1,1,y}(x)$ satisfying the recursive relation 

\[
P^{1,1,y}_n(x)=2xP^{1,1,y}_{n-1}(x)-yP^{1,1,y}_{n-2}(x)
\]
for $n\geq 2$ with initials $P^{1,1,y}_0(x)=1$ and  $P^{1,1,y}_1(x)=2x$ and admitting 
the closed-form representation

\[
P^{1, 1, y}_n(x)= \frac{(x+\sqrt{x^2-y})^{n+1}-(x-\sqrt{x^2-y})^{n+1}}{2\sqrt{x^2-y}}.
\]
\end{proposition}

\begin{example}\label{ex:2}
We may use recurrence relation (\ref{eq:14}) to define various polynomials that were defined using different techniques. Comparing recurrence relation  (\ref{eq:14}) with the relations of the generalized Fibonacci and Lucas polynomials,
with the assumption of $P^{1,C,y}_0=0$ and $P^{1,C,y}_1=1$, we immediately know 
\[
P^{1, 1, 1}_n(x)=2x P^{1,1,1}_{n-1}(x)-P^{1,1,1}_{n-2}(x)
\]
defines the Chebyshev polynomials of the second kind,  
\[
P^{1,-1,1}_n(x)=2x P^{1,-1,1}_{n-1}(x)+P^{1,-1,1}_{n-2}(x)
\]
defines the Pell polynomials, and 
\[
P^{1,-1,1}_n\left( \frac{x}{2}\right) =x P^{1,-1,1}_{n-1}\left( \frac{x}{2}\right)+P^{1,-1,1}_{n-2}\left( \frac{x}{2}\right)
\]
defines the Fibonacci polynomials. 

In addition, in \cite{LMT}, Lidl, Mullen, and Turnwald defined the Dickson polynomials are also the special case of the generalized Gegenbauer-Humbert polynomials, which can be defined uniformly using recurrence relation (\ref{eq:14}), namely  
\[
D_n(x;a))=x D_{n-1}(x;a)-a D_{n-2}(x;a) =2P^{1,2a,2}_n(x)
\]
with $D_0(x;a)=2$ and $D_1(x;a)=x$. Thus, the general terms of all of above polynomials can be expressed using (\ref{eq:11}). 
\end{example}

\begin{example}\label{ex:3}
For $\lambda =y=C=1$, using (\ref{eq:11}) we obtain the expression of the Chebyshev polynomials of the second kind:
\[
U_n(x)= \frac{(x+\sqrt{x^2-1})^{n+1}-(x-\sqrt{x^2-1})^{n+1}}{2\sqrt{x^2-1}}, ~  
n\geq 0,
\]
where $x^2\not= 1$, and $U_0(x) =1,\, U_1(x) = 2x,\,U_2(x)=4x^2-1$, etc.

For $\lambda =C=1$ and $y=-1$, formula (\ref{eq:11}) gives the expression of a 
Pell polynomial of degree $n+1$: 
\[
P_{n+1}(x)= \frac{(x+\sqrt{x^2+1})^{n+1}-(x-\sqrt{x^2+1})^{n+1}}{2\sqrt{x^2+1}}, ~ 
n\geq 0.
\]
Thus, $P_0(x) = 0,\,P_1(x) = 1,\,P_2(x)=2x,\,P_3(x)=4x^2 + 1$, etc.

Similarly, let $\lambda =C=1$ and $y=-4$, the Fibonacci polynomials are
\[
F_{n+1}(x)=\frac{(x+\sqrt{x^2+4})^{n+1}-(x-\sqrt{x^2+4})^{n+1}}{2^{n+1}\sqrt{x^2+4}}, ~ 
n\geq 0,
\]
and $F_0(x)=0, F_1(x)=1, F_2(x)=x$, etc., and the Fibonacci numbers are 
\[
F_{n}=F_{n}(1)= \frac{1}{\sqrt{5}}\left\{ \left( \frac{1+\sqrt{5}}{2}\right)^{n} - 
\left( \frac{1-\sqrt{5}}{2}\right)^{n}
\right\}.
\]

Finally, for $\lambda =C=1$ and $y=2$, we have Fermat polynomials of the first kind:
\[
\Phi_{n+1}(x)= \frac{(x+\sqrt{x^2-8})^{n+1}-(x-\sqrt{x^2-8})^{n+1}}{2^{n+1}\sqrt{x^2-8}}, 
~ n\geq 0, 
\]
where $x^2\not= 2$. From the expressions of Chebyshev polynomials of the second kind, Pell polynomials, and Fermat polynomials of the first kind, we may get a class of the generalized Gegenbauer-Humbert polynomials with respect to $y$ defined in the next section.. 
\end{example}

\section{Riordan Characterization of Polynomial Sequences Defined by
Second-Order Recurrences}

We characterize those polynomial sequences satisfying a homogeneous second-order recurrence whose coefficient matrices are ordinary Riordan arrays.

\begin{theorem}\label{thm:Riordan-characterization}
Let $\{p_n(x)\}_{n\ge0}$ be a polynomial sequence satisfying
\begin{equation}
p_{n+2}(x)=a(x)p_{n+1}(x)+b(x)p_n(x),
\qquad n\ge0,
\label{eq:general-recurrence}
\end{equation}
with initial polynomials $p_0(x)$ and $p_1(x)$, where $a(x)$ and $b(x)$ are fixed polynomials (independent of $n$). 

Define the ordinary generating function
\[
P(t,x)=\sum_{n=0}^{\infty}p_n(x)t^n.
\]

Then
\[
P(t,x)=\frac{p_0(x)+\bigl(p_1(x)-a(x)p_0(x)\bigr)t}{1-a(x)t-b(x)t^2}.
\]

The coefficient matrix of $\{p_n(x)\}$, $(p_{n,k}(x))_{n\geq k\geq 0}$, is an ordinary Riordan array $(g,f)$, i.e., 

\[
p_n(x)=\sum^n_{k=0} d_{n,k} x^k,
\]
where $d_{n,k}=[t^n] g f^k$, if and only if the following conditions hold:

\begin{enumerate}
\item
Both $a(x)$ and $b(x)$ are affine functions of $x$,
\[
a(x)=a_0+a_1x, \qquad b(x)=b_0+b_1x;
\]
\item
$p_0(x)$ is a constant,
\[
p_0(x)=c;
\]
\item
The polynomial
\[
p_1(x)-a(x)p_0(x)
\]
is independent of $x$.
\end{enumerate}

Under these conditions,

\[
P(t,x)=\frac{g(t)}{1-xf(t)},
\]
where

\begin{equation}
g(t)=\frac{c+\bigl(d-ca_0\bigr)t}{1-a_0t-b_0t^2},
\label{eq:g-characterization}
\end{equation}
and

\begin{equation}
f(t)=\frac{a_1t+b_1t^2}{1-a_0t-b_0t^2},
\label{eq:f-characterization}
\end{equation}
where 

\[
d=p_1(x)-a(x)p_0(x),
\]
which is a constant by assumption.
\end{theorem}

\begin{proof}
From Shapiro et al. \cite[Theorem 3.4]{Shapiro}, we know that a matrix is a Riordan array $(g,f)$, where $g\in \F_0$ and $f\in \F_1$, if and only if its bivariate generating function $P(t,x)$ can be written as $g(t)/(1-xf(t))$.

The generating function derived from \eqref{eq:general-recurrence} is

\[
P(t,x)=\frac{p_0(x)+\bigl(p_1(x)-a(x)p_0(x)\bigr)t}{1-a(x)t-b(x)t^2}.
\]

Suppose first that the coefficient matrix is an ordinary Riordan array. Then, from Shapiro et al. \cite[Theorem 3.4]{Shapiro}, there exist formal power series $g(t)$ and $f(t)$, independent of $t$, such that

\[
P(t,x)=\frac{g(t)}{1-xf(t)}.
\]

Consequently, there exists a nonzero formal power series
$\phi(t)$ satisfying

\[
1-a(x)t-b(x)t^2=\phi(t)\bigl(1-xf(t)\bigr).
\]

Since the right-hand side is affine in $x$, the left-hand side must
also be affine in $x$. Therefore

\[
a(x)=a_0+a_1x, \qquad b(x)=b_0+b_1x.
\]

Substituting these expressions yields

\[
1-a(x)t-b(x)t^2=\bigl(1-a_0t-b_0t^2\bigr)\left(1-x\frac{a_1t+b_1t^2}{1-a_0t-b_0t^2}
\right),
\]

so that

\[
\phi(t)=1-a_0t-b_0t^2,
\]

and

\[
f(t)=\frac{a_1t+b_1t^2}{1-a_0t-b_0t^2}.
\]

Next,

\[
p_0(x)+\bigl(p_1(x)-a(x)p_0(x)\bigr)t=g(t)\phi(t).
\]

Since the right-hand side is independent of $x$, differentiating with
respect to $x$ gives

\[
\frac{d}{dx}p_0(x)=0,
\]
and

\[
\frac{d}{dx}\left(p_1(x)-a(x)p_0(x)\right)=0.
\]

Hence

\[
p_0(x)=c
\]
is constant, and

\[
p_1(x)-a(x)p_0(x)=d
\]
is also constant.

Conversely, assume that the three stated conditions hold. Then

\[
1-a(x)t-b(x)t^2=\bigl(1-a_0t-b_0t^2\bigr)\left(1-x\frac{a_1t+b_1t^2}{1-a_0t-b_0t^2}\right),
\]
while

\[
p_0(x)+\bigl(p_1(x)-a(x)p_0(x)\bigr)t=c+dt.
\]
Therefore,

\[
P(t,x)=\frac{c+dt}{1-a_0t-b_0t^2}\cdot\frac{1}{1-x\dfrac{a_1t+b_1t^2}{1-a_0t-b_0t^2}}.
\]
Thus

\[
P(t,x)=\frac{g(t)}{1-xf(t)},
\]
where $g(t)$ and $f(t)$ are given by \eqref{eq:g-characterization} and \eqref{eq:f-characterization}. 
Hence, from Shapiro et al. \cite[Theorem 3.4]{Shapiro}, the coefficient matrix is an ordinary Riordan array. This completes the proof.
\end{proof}

Clearly, the Riordan array representation $(g,f)$ gives the coefficient matrix of the recursive polynomial sequence $\{ p_n(x)\}_{n\geq 0}$. If we apply the FTRA to $(g,f)$ and the generating function $1/(1-xt)$ of the sequence $\{x^n\}_{n\geq 0}$, we obtain the G.F. of $\{p_n (x)\}_{n\geq 0}$. 

\[
(g(t),f(t))\frac{1}{1-xt}=\frac{g(t)}{1-xf(t)}=\sum_{n\geq 0}p_n(x)t^n.
\]
Therefore, from the general Binet formula for second-order linear recurrences, we obtain the following result.

\begin{proposition}\label{pro:RC-2}
Let $\{ p_n(x)\}_{n\geq 0}$ be the recursive polynomial sequence defined in Theorem \ref{thm:Riordan-characterization} with a Riordan array representation $(g,f)$. Then there exists the identity:

\begin{equation}\label{0-0-2}
\left(\frac{c+\bigl(d-ca_0\bigr)t}{1-a_0t-b_0t^2},
\frac{a_1t+b_1t^2}{1-a_0t-b_0t^2}\right) \frac{1}{1-xt}=(p_0(x), p_1(x), \ldots)^T,
\end{equation}
where $c=p_0(x)$, $d=p_1(x)-a(x)p_0(x)$, 

\[
p_n(x)=\left\{ \begin{array}{ll} \left( \frac{p_1(x)-\beta(x) p_0(x)}{\alpha (x)-\beta (x)}\right) \alpha^n(x)- \left(\frac{p_1(x)-\alpha(x) p_0(x)}{\alpha(x)-\beta(x)}\right) \beta^n(x), & if\,\, \alpha(x)\not= \beta(x);\\
np_1(x) \alpha^{n-1}(x)-(n-1)p_0(x)\alpha^{n}(x), &if\,\, \alpha (x)=\beta(x),\end{array}\right.,
\]
and $\alpha(x)$ and $\beta(x)$ are shown below:

\[
\alpha(x)= \frac{1}{2}( a(x)+\sqrt{a^2(x)+4b(x)}), \beta (x)= \frac{1}{2}( a(x)-\sqrt{a^2(x)+4b(x)}).
\]
\end{proposition}

\begin{proof}
From Theorem \ref{thm:Riordan-characterization}, generating function (\ref{eq:10}) of the 
sequence $\{P^{\lambda, C}_n (x,y)\}$ together with Definition \ref{def:5} suggest a presentation of the sequence 
$\{ P^{1,C,y}_n(x)\}_{n\geq 0}$ shown in \eqref{eq:14} by a Riordan array. More precisely, if we assume $C\neq 0$, then 
\[
a(x)=\frac{2x}{C},\,\, b(x)=-\frac{y}{C},\,\, P^{1,C}_0(x,y)=P^{1,C,y}_0(x)=\frac{1}{C},\,\, P^{1,C}_1(x,y)=P^{1,C,y}_1(x)=\frac{2x}{C^2}
\]
implies 

\[
a_0=0,\,\, b_1=0,\,\, c=\frac{1}{C},\,\, d=0.
\]
Thus, from Theorem \ref{thm:Riordan-characterization} and noting $a_0=0$, $b_1=0$, $p_0(x)=1/C$, $p_1(x)=2x/C^2$, and 
$d=0$, the generating function of $P^{1,C,y}_n(x)$ 

\begin{equation}\label{0-0-1}
P(t,x)=\sum_{n\geq 0}P^{1,C,y}_n(x)t^n=\frac{g(t)}{1-xf(t)},
\end{equation}
where 

\[
g(t)=\frac{t}{C+yt^2}\quad f(t)=\frac{2x}{C+yt^2}.
\]
Thus sequence $\{P^{1,C,y}_n(x)\}_{n\geq 0}$ has Riordan array representation 
\begin{equation}
\label{RA1}
(g,f)=\left( \frac{1}{C+yt^2}, \, \frac{2t}{C+yt^2}\right),
\end{equation}
which is a proper Riordan array and begins with  
\begin{equation}
\label{MRA1}
\begin{pmatrix}
\frac{1}{C} & 0 & 0 & 0 & 0 & \cdots \\
0 & \frac{2}{C^2} & 0 & 0 & 0 & \cdots \\
\frac{-y}{C^2} & 0 & \frac{4}{C^3} & 0 & 0 & \cdots \\
0 & \frac{-4y}{C^3} & 0 & \frac{8}{C^4} & 0 & \cdots \\
\frac{y^2}{C^3} & 0 & \frac{-12y}{C^4} & 0 & \frac{16}{C^5} & \cdots \\
\vdots & \vdots & \vdots & \vdots & \vdots & \ddots\\
\end{pmatrix}.
\end{equation}
Since $(g,f)$ is the coefficient matrix of $\{p_n(x)\}_{n\geq 0}$, we get \eqref{0-0-2}. Additionally, the generating function of 
$\{ P_n^{1,C,y}(x)\}$ is 

\begin{equation}\label{P(t,x)}
P(t,x)=\left( \frac{1}{C+yt^2}, \, \frac{2t}{C+yt^2}\right)\frac{1}{1-xt}=\frac{1}{C-2xt+yt^2}.
\end{equation}
\end{proof}

\begin{example}
From Example \ref{ex:3} and using Proposition \ref{pro:RC-2}, we obtain the identity for the Chebyshev polynomials of the second kind: 

\[
\left(\frac{1}{1+t^2},
\frac{2t}{1+t^2}\right) \frac{1}{1-xt}=(U_0(x), U_1(x), \ldots)^T,
\]
where 
\[
U_n(x)= \frac{(x+\sqrt{x^2-1})^{n+1}-(x-\sqrt{x^2-1})^{n+1}}{2\sqrt{x^2-1}}, ~  
n\geq 0,
\]
where $x^2\not= 1$.

Similarly, for Pell polynomial of degree $n+1$: 
\[
P_{n+1}(x)= \frac{(x+\sqrt{x^2+1})^{n+1}-(x-\sqrt{x^2+1})^{n+1}}{2\sqrt{x^2+1}}, ~ 
n\geq 0,
\]
we have the identity 

\[
\left(\frac{1}{1-t^2},
\frac{2t}{1-t^2}\right) \frac{1}{1-xt}=(P_0(x), P_1(x), \ldots)^T.
\]
\end{example}

Some properties of a recursive polynomial sequence can be derived by using the Riordan array representation of the coefficient matrix of a recursive polynomial sequence. For instance, we have the following result. 

\begin{proposition}\label{pro:BC-3}
Let $\{ p_n(x)\}_{n\geq 0}$ be the recursive polynomial sequence defined in Theorem \ref{thm:Riordan-characterization} with a Riordan array representation $(g,f)$. Then,

\[
\sum^n_{k=0}[x^k]p_{2n+1-k}(x)=0,\quad \sum^n_{k=0}[x^k]p_{2n-k}(x)=\frac{(2x-y)^n}{C^{n+1}}.
\]
\end{proposition}

\begin{proof}
Denote $(d_{n,k})_{n\geq k\geq 0}=(g,f)$. Since $[x^k]p_{2n+1-k}(x)=d_{2n+1-k,k}$ and $[x^k]p_{2n-k}(x)=d_{2n-k,k}$, we only need to prove for every natural $n\geq 0$, there exist 
\begin{equation}
\label{RisSum}
\sum\limits_{k=0}^{n} d_{2n+1-k,k} =0 \quad \mbox{and} \quad
\sum\limits_{k=0}^{n} d_{2n-k,k}=\frac{(2x-y)^n}{C^{n+1}}.
\end{equation}
We prove the second equality of (\ref{RisSum}), and leave the first one for 
the reader. According to the FTRA, elements of the $k$-th column of the 
Riordan array (\ref{RA1}) are the coefficients of the f.p.s.
$$
\sum\limits_{i=0}^{\infty} d_{i,k} t^i=\frac{(2xt)^k}{(C+yt^2)^{k+1}}.
$$
Therefore, 
$$
d_{2n-k,k}=\bigl[t^{2n-k}\bigr]\frac{(2xt)^k}{(C+yt^2)^{k+1}}= 
\frac{(2x)^k}{C^{k+1}} \bigl[t^{2n-k}\bigr] \frac{t^k}{\bigl(1+(y/C)t^2\bigr)^{k+1}}.
$$
Using the {\sl Newton's rule}
$$
\bigl[t^n \bigr] (1+\alpha t)^r={r \choose n}\alpha^n
$$
together with the basic rules of operation of the operator $[t^n]$, we obtain 
$$
d_{2n-k,k}=\frac{(2x)^k}{C^{k+1}} \bigl[t^{2(n-k)}\bigr] \bigl(1+(y/C)t^2\bigr)^{-(k+1)}
= \frac{(2x)^k}{C^{k+1}} {-(k+1) \choose n-k} \frac{y^{n-k}}{C^{n-k}}
$$
\begin{equation}
\label{DiagEl}
= \frac{(2x)^k y^{n-k}}{C^{n+1}} {n-k+(k+1)-1 \choose n-k}(-1)^{n-k}=
\frac{1}{C^{n+1}} {n \choose k} (2x)^k (-y)^{n-k} .
\end{equation}
Now the second sum of (\ref{RisSum}) follows from the binomial theorem 
$$
\sum\limits_{k=0}^{n} d_{2n-k,k}=\sum\limits_{k=0}^{n} 
\frac{1}{C^{n+1}} {n \choose k} (2x)^k (-y)^{n-k}=\frac{(2x-y)^n}{C^{n+1}}.
$$
\end{proof}

For the polynomial sequence defined in Theorem \ref{thm:Riordan-characterization} with a Riordan array representation $(g,f)$ has the bivariate generating function $P(t,x)=g(t)/(1-xf(t))$, we get 
\[
p_n(x)=[t^n]P(t,x)=[t^n]\frac{g(t)}{1-xf(t)}.
\]

\begin{proposition}\label{pro:BC-4}
Let $\{ p_n(x)\}_{n\geq 0}$ be the recursive polynomial sequence defined in Theorem \ref{thm:Riordan-characterization} with a Riordan array representation $(g,f)$. Then,
\begin{align}
&\frac{d^\ell}{dx^\ell}p_n(1)=[t^n]\frac{g(t)f(t)^{\ell-1}}{\ell ! (1-f(t))^{\ell+1}}\label{0-0-3}\\
&\frac{d^\ell}{dx^\ell}p_n(-1)=[t^n]\frac{g(t)f(t)^{\ell-1}}{\ell !(1+f(t))^{\ell+1}}\label{0-0-4}.
\end{align}
Particularly, 

\begin{align*}
&p_n(1)=[t^n]\frac{g(t)}{1-f(t)},\\
&p_n(-1)=[t^n]\frac{g(t)}{1+f(t)},\\
&p'_n(1)=[t^n]\frac{g(t)f(t)}{(1-f(t))^2}\\
&p'_n(-1)=-[t^n]\frac{g(t)f(t)}{(1+f(t))^2}.
\end{align*}
\end{proposition}

\begin{proof}
Since 

\[
\frac{d^\ell}{dx^\ell}p_n(x)=[t^n]\frac{g(t)f(t)^{\ell-1}}{\ell ! (1-xf(t))^{\ell+1}},
\]
we immediately obtain the results. 
\end{proof}

We may consider $P^{1,C,y}_n(x)$ as a polynomial of $y$ with respect to the parameters $C$ and $x$, denoted by $\{P^{1,C,x}_n(y)\}$, which satisfies linear recurrence relation 
\be\label{eq:14-y}
P^{1,C,x}_n(y)=\frac{2x}{C}P^{1,C,x}_{n-1}(y)-\frac{y}{C}P^{1,C,x}_{n-2}(y), \quad n\geq 2,
\ee
and 
\begin{equation}\label{0-0-y}
P^{1,C,x}_0(y) =C^{-1},\quad P^{1,C,x}_1(y)=2x C^{-2}. 
\end{equation}

From Theorem \ref{thm:Riordan-characterization}, and noting $a(y)=2x/C$ and $b(y)=-y/C$, which gives 

\begin{align*}
&a_0=\frac{2x}{C}, \,\, a_1=0,\,\, b_0=0, \,\, b_1=-\frac{1}{C}, \\
& p_0(y)=\frac{1}{C},\,\, p_1(y)=\frac{2x}{C^2},\,\, d=p_1(y)-a(y)p_0(y)=0,
\end{align*}
we obtain the Riordan array representation of $\{P^{1,C,x}_n(y)\}$ as

\[
\left( g(t), f(t)\right)=\left( \frac{\frac{1}{C}}{1-\frac{2x}{C}t}, \frac{-\frac{t^2}{C}}{1-\frac{2x}{C}t}\right)=\left(\frac{1}{C-2xt}, \frac{-t^2}{C-2xt}\right).
\]
Therefore, the generating function of $\{P^{1,C,x}_n(y)\}$ is

\[
P(t,y)=\left(\frac{1}{C-2xt}, \frac{-t^2}{C-2xt}\right)\frac{1}{1-yt}=\frac{1}{C-2xt+yt^2},
\]
which equals to $P(t,x)$ shown in \eqref{P(t,x)}.

Yet another way to represent the sequence $\{P^{1, C,x}_n (y)\}_{n\geq 0}$, 
via a Riordan array, is to use powers of $x$ in the coefficients array. For example, 
when $C\neq 0$ Nikolai Krylov suggests to take Riordan array 
\begin{equation}
\label{RA3-2}
\left( \frac{1}{C-2xt}, \, \frac{t}{\sqrt{C-2xt}}\right).
\end{equation}
Indeed, applying the FTRA to the aerated sequence of powers of $-y$ one gets 
\begin{equation}
\label{RA3-2-2}
\left( \frac{1}{C-2xt}, \, \frac{t}{\sqrt{C-2xt}}\right)\cdot \frac{1}{1+yt^2} =
\frac{1}{C-2xt}\cdot \frac{1}{1+\frac{yt^2}{C-2xt}}= \frac{1}{C-2xt+yt^2}.
\end{equation}

\section{Riordan Characterization for Higher-Order Polynomial Recurrences}

In this section, we investigate when the generating function associated with a polynomial sequence satisfying a homogeneous linear recurrence of order $\ell$ can be represented by an ordinary Riordan array.

Let $\{p_n(x)\}_{n\ge0}$ satisfy
\begin{equation}
p_{n+\ell}(x)=\sum_{i=1}^{\ell}a_i(x)p_{n+\ell-i}(x),\qquad n\ge0,
\label{eq:l-recurrence}
\end{equation}
where each $a_i(x)$ is a polynomial in $x$. We call $\{p_n(x)\}_{n\ge0}$ a recursive polynomial sequence (RPS) of order $\ell$. 

Let
\be\label{3.2}
P(t,x)=\sum_{n=0}^{\infty}p_n(x)t^n
\ee
denote the ordinary generating function.

By Theorem~\ref{thm:lth-order-generating-function},
\begin{equation}
P(t,x)=\frac{N(t,x)}{D(t,x)},
\label{eq:GF-general}
\end{equation}
where
\[
D(t,x)=1-\sum_{i=1}^{\ell}a_i(x)t^i,
\]
and
\[
N(t,x)=\sum_{m=0}^{\ell-1}\left(p_m(x)-\sum_{i=1}^{m}a_i(x)p_{m-i}(x)
\right)t^m.
\]

Recall that a pair $(g(t),f(t))$ defines an ordinary Riordan array if
\begin{equation}
g(0)\neq 0, \qquad f(0)=0, \qquad f'(0)\neq 0.
\label{eq:riordan-conditions}
\end{equation}

Its bivariate generating function is
\begin{equation}
P(t,x) =\frac{g(t)}{1-xf(t)}.
\label{eq:riordan-generating-function}
\end{equation}

The following theorem gives the exact criterion for a rational
generating function obtained from a recurrence to have this form.

\begin{theorem}
\label{thm:exact-riordan}
Let
\[
P(t,x)=\frac{N(t,x)}{D(t,x)}
\]
be the generating function of the RPS $\{p_n(x)\}$ given in Theorem \ref{thm:lth-order-generating-function}, with $D(0,x)=1$.

Then the following statements are equivalent.
\begin{enumerate}
\item
There exist formal power series $g(t)$ and $f(t)$, independent of $x$,
such that
\begin{equation}
P(t,x) =\frac{g(t)}{1-xf(t)},
\label{eq:desired-riordan-form}
\end{equation}
where
\[
g(0)\neq 0, \qquad f(0)=0, \qquad f'(0)\neq 0.
\]
\item There exist formal power series
$q(t,x)$, $N_0(t)$, $D_0(t)$, and $D_1(t)$ such that
\begin{align}
N(t,x) &=q(t,x)N_0(t), \label{eq:N-factorization} \\
D(t,x) &= q(t,x)\bigl(D_0(t)+xD_1(t)\bigr),
\label{eq:D-factorization}
\end{align}
where
\[
N_0(0)\neq 0, \qquad D_0(0)\neq 0,
\]
and
\begin{equation}
D_1(0)=0, \qquad D_1'(0)\neq 0.
\label{eq:D1-conditions}
\end{equation}

In this case,
\begin{equation}
g(t)=\frac{N_0(t)}{D_0(t)}
\label{eq:g-formula}
\end{equation}
and
\begin{equation}
f(t)=-\frac{D_1(t)}{D_0(t)}.
\label{eq:f-formula}
\end{equation}
\end{enumerate}
\end{theorem}

\begin{proof}
Suppose first that statement (1) holds. Then
\[
\frac{N(t,x)}{D(t,x)}=\frac{g(t)}{1-xf(t)}.
\]

Define
\[
q(t,x)=\frac{D(t,x)}{1-xf(t)}.
\]

Then
\[
D(t,x)=q(t,x)\bigl(1-xf(t)\bigr)
\]
and
\[
N(t,x)=q(t,x)g(t).
\]

Thus statement (2) holds with
\[
N_0(t)=g(t), \qquad D_0(t)=1, \qquad D_1(t)=-f(t).
\]

Since
\[
g(0)\neq 0, \qquad f(0)=0, \qquad f'(0)\neq 0,
\]
we obtain
\[
N_0(0)\neq 0, \qquad D_0(0)\neq 0, \qquad D_1(0)=0, \qquad D_1'(0)\neq 0.
\]

Therefore statement (2) follows.

Conversely, suppose statement (2) holds. Then
\begin{align*}
P(t,x) &= \frac{q(t,x)N_0(t)} {q(t,x)\bigl(D_0(t)+xD_1(t)\bigr)}\\
&= \frac{N_0(t)} {D_0(t)+xD_1(t)}.
\end{align*}

Since
\[
D_0(0)\neq 0,
\]
the formal power series $D_0(t)$ is invertible. Define
\[
g(t)=
\frac{N_0(t)}{D_0(t)}
\]
and
\[
f(t)=-\frac{D_1(t)}{D_0(t)}.
\]

Then
\[
D_0(t)+xD_1(t)=D_0(t) \left( 1-xf(t)\right).
\]

Consequently,
\begin{align*}
P(t,x) &= \frac{N_0(t)} {D_0(t)(1-xf(t))}\\
&= \frac{g(t)}{1-xf(t)}.
\end{align*}

Furthermore,
\[
g(0)=\frac{N_0(0)}{D_0(0)}
\neq 0.
\]

Since $D_1(0)=0$, we have
\[
f(0)=0.
\]

Finally,
\[
f'(0)=-\frac{D_1'(0)}{D_0(0)} \neq 0,
\]
because $D_1'(0)\neq 0$ and $D_0(0)\neq 0$.

Thus $(g,f)$ is an ordinary Riordan array, and statement (1)
follows.
\end{proof}

\begin{corollary}\label{cor:3.2}
Let
\[
P(t,x)=\frac{N(t,x)}{D(t,x)}
\]
be the generating function of the RPS $\{p_n(x)\}$ given in Theorem~\ref{thm:lth-order-generating-function}, with 
\[
D(0,x)=1.
\]
Then, $\{p_n(x)\}$ has a Riordan representation $(g,f)$ if and only if, after cancellation of a common factor, the  numerator $N(t,x)$ depends only on $t$ and the denominator $D(t,x)$ is affine in $x$. More precisely, if
\[
N(t,x)=q(t,x)N_0(t)
\]
and
\[
D(t,x)=q(t,x)\bigl(D_0(t)+xD_1(t)\bigr),
\]
then
\[
g(t)=\frac{N_0(t)}{D_0(t)}
\]
and
\[
f(t)=-\frac{D_1(t)}{D_0(t)},
\]
where pair $(g,f)$ is an ordinary Riordan array precisely when
\[
g(0)\neq0, \qquad f(0)=0, \qquad f'(0)\neq0.
\]
\end{corollary}

\begin{proof} From Theorem \ref{thm:exact-riordan}, the given conditions derive there exist formal power series $g(t)$ and $f(t)$, independent of $x$, such that
\[
P(t,x) =\frac{g(t)}{1-xf(t)},
\]
where
\[
g(0)\neq 0, \qquad f(0)=0, \qquad f'(0)\neq 0.
\]
\end{proof}

A particularly useful sufficient condition for the Riordan representation is that this numerator of $P(t,x)$ be independent of $x$.

\begin{theorem}\label{thm:3.3}
Let
\[
P(t,x)=\frac{N(t,x)}{D(t,x)}
\]
be the generating function of the RPS $\{p_n(x)\}$ given in Theorem~\ref{thm:lth-order-generating-function}, with 
\[
D(0,x)=1.
\]
Suppose that the numerator is independent of $x$,
\[
N(t,x)=N_0(t), \quad N_0(0)=1, 
\]
and the recurrence coefficients of $\{p_n(x)\}$ shown in \eqref{eq:l-recurrence} are affine functions of $x$:
\begin{equation}
a_i(x)=\alpha_i+\beta_i x, \qquad
1\leq i\leq\ell,
\label{eq:affine-recurrence}
\end{equation}
with $\beta_1\not=0$. 
Then, RPS $\{p_n(x)\}$ has a Riordan array representation $(g,f)$, where 

\begin{equation}
f(t)=\frac{ \displaystyle\sum_{i=1}^{\ell}\beta_i t^i}
{\displaystyle 1-\sum_{i=1}^{\ell}\alpha_i t^i}
\label{eq:f-affine}
\end{equation}
with $f(0)=0$ and $f'(0)\not=0$.
\end{theorem}

\begin{proof} If the recurrence coefficients of $\{p_n(x)\}$ shown in \eqref{eq:l-recurrence} are affine functions of $x$, $a_i(x)=\alpha_i+\beta_i x$, $1\leq i\leq\ell$. Then the denominator of $P(t,x)$ can be written as 

\begin{align}
D(t,x) &= 1-\sum_{i=1}^{\ell}(\alpha_i+\beta_i x)t^i \nonumber\\
&=1-\sum_{i=1}^{\ell}\alpha_i t^i-x\sum_{i=1}^{\ell}\beta_i t^i.
\label{eq:affine-denominator}
\end{align}

Thus, 
\begin{equation}
D(t,x)=D_0(t)+xD_1(t),
\label{eq:affine-D}
\end{equation}
where
\begin{equation}
D_0(t)=1-\sum_{i=1}^{\ell}\alpha_i t^i
\label{eq:D0}
\end{equation}
and
\begin{equation}
D_1(t)=-\sum_{i=1}^{\ell}\beta_i t^i.
\label{eq:D1}
\end{equation}

If the numerator and denominator have no common factor that changes
the reduced denominator, then Theorem~\ref{thm:exact-riordan} gives
\begin{equation}
f(t)=\frac{ \displaystyle\sum_{i=1}^{\ell}\beta_i t^i}
{\displaystyle 1-\sum_{i=1}^{\ell}\alpha_i t^i}.
\label{eq:f-affine}
\end{equation}

Since
\[
f(t)=\beta_1t+O(t^2),
\]
we have
\begin{equation}
f(0)=0, \qquad f'(0)=\beta_1\neq 0.
\label{eq:f-prime-affine}
\end{equation}

If, in addition, the numerator is independent of $x$,
\[
N(t,x)=N_0(t) 
\]
with $N_0(0)=1$, then
\begin{equation}
g(t)=\frac{N_0(t)}{\displaystyle1-\sum_{i=1}^{\ell}\alpha_i t^i},
\label{eq:g-affine}
\end{equation}
where $g(0)=1$. Therefore, $(g,f)$ is the Riordan array represensentation of $\{p_n(x)\}_{n\geq 0}$. 
\end{proof}

\begin{corollary}\label{cor:3.4}
Let
\[
P(t,x)=\frac{N(t,x)}{D(t,x)}
\]
be the generating function of the RPS $\{p_n(x)\}$ given in Theorem~\ref{thm:lth-order-generating-function}, with 
\[
D(0,x)=1.
\]
Suppose that the recurrence coefficients of $\{p_n(x)\}$ shown in \eqref{eq:l-recurrence} are affine functions of
$x$:
\begin{equation}
a_i(x)=\alpha_i+\beta_i x, \qquad
1\leq i\leq\ell,
\label{eq:affine-recurrence}
\end{equation}
where $\beta_1\not= 0$. 
If, in addition, there exist constants
$c_0\not= 0,c_1,\ldots,c_{\ell-1}$ such that
\begin{equation}
p_j(x)-\sum_{i=1}^{j}
(\alpha_i+\beta_i x)p_{j-i}(x)=c_j,
\qquad
0\leq j\leq\ell-1.
\label{eq:initial-compatibility}
\end{equation}
Then, RPS $\{p_n(x)\}$ has a Riordan array representation $(g,f)$, where 

\be
g(t)=\frac{\displaystyle\sum_{j=0}^{\ell-1}c_jt^j}
{\displaystyle 1-\sum_{i=1}^{\ell}\alpha_i t^i}, \qquad f(t)=\frac{ \displaystyle\sum_{i=1}^{\ell}\beta_i t^i}
{\displaystyle 1-\sum_{i=1}^{\ell}\alpha_i t^i},
\label{eq:g-initial}
\ee
with $g(0)\neq 0$, $f(0)=0$, and $f'(0)\neq 0$.
\end{corollary}

\begin{proof} From Theorem \ref{thm:3.3}, and notice 
\[
N(t,x)=\sum_{j=0}^{\ell-1}c_jt^j=:N_0(t),
\]
which is independent of $x$, and 
\[
N_0(0)\not= 0. 
\]
Hence, RPS $\{p_n(x)\}$ has a Riordan array representation $(g,f)$, where $f$ is given in 
Theorem \ref{thm:3.3}, and 
\begin{equation}
g(t)=\frac{\displaystyle\sum_{j=0}^{\ell-1}c_jt^j}
{\displaystyle 1-\sum_{i=1}^{\ell}\alpha_i t^i}.
\label{eq:g-initial}
\end{equation}

Together with
\[
c_0\neq 0 \qquad \beta_1\neq0,
\]
this gives an ordinary Riordan array.
\end{proof}

\begin{example}
As an example, we consider a generalized Tribonacci polynomial sequence with recurrence relation
\be\label{3.3}
p_{n+3}(x)=xp_{n+2}(x)+p_{n+1}(x)+p_n(x), \quad n\geq 0,
\ee
with initials 

\[
p_0(x)=1, \qquad p_1(x)=x, \qquad p_2(x)=x^2+1.
\]
Since 

\[
a_1(x)=1, \qquad a_2(x)=1, \qquad a_3(x)=1.
\]

The first few polynomials are therefore
\begin{align}
p_0(x)&=1,\\
p_1(x)&=x,\\
p_2(x)&=x^2+1,\\
p_3(x) &=xp_2(x)+p_1(x)+p_0(x) =x^3+2x+1,\\
p_4(x) &=xp_3(x)+p_2(x)+p_1(x) =x^4+3x^2+2x+1.
\end{align}

Thus the sequence begins
\[
1,\quad x,\quad x^2+1,\quad x^3+2x+1,\quad x^4+3x^2+2x+1,\ldots
\]

The coefficients in the recurrence are affine functions of $x$. Indeed,
\[
a_i(x)=\alpha_i+\beta_i x, \qquad i=1,2,3,
\]
with
\[
\begin{array}{c|cc}
i & \alpha_i & \beta_i\\
\hline
1 & 0 & 1\\
2 & 1 & 0\\
3 & 1 & 0
\end{array}.
\]

We now examine the compatibility condition
\[
p_j(x) - \sum_{i=1}^{j} (\alpha_i+\beta_i x)p_{j-i}(x) = c_j,\qquad 0\leq j\leq2,
\]
where each $c_j$ is required to be independent of $x$.

For $j=0$, the sum is empty. Therefore,
\[
c_0=p_0(x)=1,
\]
and hence $c_0=1$. 

For $j=1$, we obtain
\begin{align}
p_1(x)-(\alpha_1+\beta_1x)p_0(x)
&=x-x(1)\\
&=0.
\end{align}
Thus, $c_1=0$. 

For $j=2$, we have
\begin{align}
& p_2(x)-(\alpha_1+\beta_1x)p_1(x)-(\alpha_2+\beta_2x)p_0(x)\\
&=(x^2+1)-x(x)-1(1)\\
&=0.
\end{align}

Therefore, $c_2=0$. 

Consequently, the compatibility condition is satisfied with
\[
c_0=1,\qquad
c_1=0,\qquad
c_2=0.
\]

The numerator of the generating function $P(t,x)$ is 

\[
N(t,x)=\sum^2_{j=0}c_jt^j=1,
\]
And the denominator of the generating function $P(t,x)$ can be written as
\[
D(t,x)=1-xt-t^2-t^3=(1-t^2-t^3)-xt.
\]

Therefore,
\begin{align}
P(t,x)&=\frac{1}{(1-t^2-t^3)-xt} \\ &= \frac{1}{1-t^2-t^3} \frac{1} {1-x\displaystyle\frac{t}{1-t^2-t^3}}.
\end{align}

Hence,
\[
P(t,x)=
\frac{g(t)}{1-xf(t)},
\]
where
\[
g(t)=\frac{1}{1-t^2-t^3},
\qquad
f(t)=\frac{t}{1-t^2-t^3}.
\]

Since
\[
g(0)=1\neq 0, \qquad f(0)=0, \qquad f'(0)=1\neq 0,
\]
the pair
\[
(g,f) = \left(\frac{1}{1-t^2-t^3},\frac{t}{1-t^2-t^3}\right)
\]
defines an ordinary Riordan array.

Thus the coefficient array associated with this polynomial sequence is
an ordinary Riordan array $(g,f)=(d_{n,k})_{n,k\geq 0}$, where 

\begin{align*}
d_{n,k} &= [t^n] \frac{1}{1-t^2-t^3} \left( \frac{t}{1-t^2-t^3}\right)^k\\
&= [t^{n-k}] (1-t^2-t^3)^{-(k+1)}.
\end{align*}

Hence the initial portion of the coefficient matrix is
\[
(g,f) =
\begin{pmatrix}
1&0&0&0&0\\
0&1&0&0&0\\
1&0&1&0&0\\
1&2&0&1&0\\
1&2&3&0&1
\end{pmatrix}.
\]
\end{example}

The first few rows can also be obtained directly from the polynomial
sequence. Since
\[
p_n(x)=\sum_{k=0}^{n}d_{n,k}x^k,
\]
we have
\[
\begin{pmatrix}
p_0(x)\\
p_1(x)\\
p_2(x)\\
p_3(x)\\
p_4(x)
\end{pmatrix}
=
\begin{pmatrix}
1\\
x\\
x^2+1\\
x^3+2x+1\\
x^4+3x^2+2x+1
\end{pmatrix}.
\]

\section{Production matrix approach of orthogonal polynomials}

Section 9.1 of \cite{Shapiro} presents a production matrix method for studying the orthogonal polynomial sequences (OPSs)  using Riordan arrays, which requires the coefficients of the three-term recurrence relation of the OPS to be constants. Such OPSs are called Riordan array-type OPSs. Therefore, this method is not applicable to general Gegenbauer-Humbert polynomial sequences. This section presents a production matrix approach applicable to any OPS. And we use Legender polynomials of non-Riordan-type OPSs and Chebyshev polynomials of Riordan-type OPSs as examples to illustrate our method.

Our production matrix representation and approach depends on the following original definition of production matrix of an infinite invertible matrix. 

\begin{definition}\label{def:PM}
Let $A$ be a lower triangular infinite invertible matrix, and let $S$ be an up-shift matrix, 

\[
S=(s_{n,k})_{n,k\ge 0}, \qquad s_{n,k}=\delta_{n,k-1},
\]
where $\delta_{i,k}$ denotes the Kronecker delta, i.e., 

\[
S=
\begin{pmatrix}
0 & 1 & 0 & 0 & 0 & \cdots \\
0 & 0 & 1 & 0 & 0 & \cdots \\
0 & 0 & 0 & 1 & 0 & \cdots \\
0 & 0 & 0 & 0 & 1 & \cdots \\
\vdots & \vdots & \vdots & \vdots & \vdots & \ddots
\end{pmatrix}.
\]
Then the matrix 

\begin{equation}\label{PM}
P_A=A^{-1}SA
\end{equation}
is called the production matrix associated with matrix $A$.
\end{definition}    

\begin{remark}
Note that the matrix $A$ in the definition is not necessarily a lower triangular matrix. Since we are interested in the coefficient matrix $A$ of the polynomial sequence, it must be a lower triangular matrix. For the sake of simplicity, we will only discuss the production matrix of the infinite lower triangular invertible matrix in this article.
\end{remark}

Every OPS $\{p_n(x)\}$ satisfies a standard three-term recurrence
\[
x\,p_n(x)=\alpha_n p_{n+1}(x)+\beta_n p_n(x)+\gamma_n p_{n-1}(x),~ n\geq 1
\]
with constant sequences $\alpha_n,\beta_n,\gamma_n$. Hence, in the $p$-basis 
(i.e. the infinite column $\{p_0(x),\,p_1(x),\,p_2(x),\,\ldots\}^T$) multiplication by 
$x$ is represented by an infinite so-called \emph{Jacobi matrix}, that is 
denoted by $J$:

\begin{equation}\label{Trid}
x\,p(x)=J\,p(x),
\end{equation}
where $J$ is tridiagonal, that has the first few rows as 

\begin{equation}\label{Jacobi}
J=
\begin{pmatrix}
\beta_0 & \alpha_0 & 0 & 0 & 0 & \cdots \\
\gamma_1 & \beta_1 & \alpha_1 & 0 & 0 & \cdots \\
0 & \gamma_2 & \beta_2 & \alpha_2 & 0 & \cdots \\
0 & 0 & \gamma_3 & \beta_3 & \alpha_3 & \cdots \\
\vdots & \vdots & \vdots & \vdots & \ddots & \ddots
\end{pmatrix}.
\end{equation}
If the recurrence coefficients in its three recurrence relations are all constants:

\[
x\,p_n(x)=\alpha p_{n+1}(x)+\beta p_n(x)+\gamma p_{n-1}(x),
\]
or equivalently,

\begin{equation}\label{Trid-2}
x\,p(x)=\tilde J\,p(x),
\end{equation}
where 

\begin{equation}\label{Jacobi-2}
\tilde J=
\begin{pmatrix}
\beta & \alpha & 0 & 0 & 0 & \cdots \\
\gamma & \beta & \alpha & 0 & 0 & \cdots \\
0 & \gamma & \beta & \alpha & 0 & \cdots \\
0 & 0 & \gamma & \beta & \alpha & \cdots \\
\vdots & \vdots & \vdots & \vdots & \ddots & \ddots
\end{pmatrix},
\end{equation}
which has a Riordan representation, while $J$ doesn't have a Riordan representation. All those will be shown below. 

\begin{example}
We can write the recurrence relation \eqref{eq:13} of the generalized 
Gegenbauer-Humbert polynomial sequence 
$\{ P_n^{\lambda,C,y}(x)\}_{n\geq 0}$ in standard form 

\begin{equation}
\label{eq:11-2}
xP_n(x)=\frac{C(n+1)}{2(\lambda +n)} P_{n+1}(x)+\frac{y(2\lambda+n-1)}{2(\lambda +n)}P_{n-1}(x), 
\end{equation}
where the recurrence coefficients are $\alpha_n=C(n+1)/(2(\lambda +n))$, 
$\beta_n=0$, and $\gamma_n=y(2\lambda +n-1)/(2(\lambda +n))$. Hence, the Jacobi matrix is 

\begin{equation}\label{Jacobi-3}
J=
\begin{pmatrix}
0 & \frac{C}{2\lambda} & 0 & 0 & 0 & \cdots \\
\frac{2\lambda y}{2(\lambda +1)} & 0 & \frac{2C}{2(\lambda +1)}& 0 & 0 & \cdots \\
0 & \frac{y(2\lambda +1)}{2(\lambda +2)} & 0 & \frac{3C}{2(\lambda +2)} & 0 & \cdots \\
0 & 0 & \frac{y(2\lambda +2)}{2(\lambda +3)} & 0 & \frac{4C}{2(\lambda+3)} & \cdots \\
\vdots & \vdots & \vdots & \vdots & \ddots & \ddots
\end{pmatrix}.
\end{equation}
If $\lambda =1$, then \eqref{eq:11-2} can be derived as \eqref{eq:14}, 
whose standard form is: 

\begin{equation}\label{eq:14-2}
xP_n(x)=\frac{C}{2}P_{n+1}(x)+\frac{y}{2}P_{n-1}(x),
\end{equation}
where the recurrence coefficients are $\alpha =C/2$, $\beta =0$, and 
$\gamma =y/2$, and the corresponding Jacobi matrix is 

\[
\tilde J=
\begin{pmatrix}
0 & \frac{C}{2} & 0 & 0 & 0 & \cdots \\
\frac{y}{2} & 0 & \frac{C}{2}& 0 & 0 & \cdots \\
0 & \frac{y}{2} & 0 & \frac{C}{2} & 0 & \cdots \\
0 & 0 & \frac{y}{2} & 0 & \frac{C}{2} & \cdots \\
\vdots & \vdots & \vdots & \vdots & \ddots & \ddots
\end{pmatrix}.
\]
\end{example}

If $A$ is the (lower-triangular) coefficient matrix of the OPS $\{p_n(t)\}_{n\geq 0}$, i.e., row $n$ of $A$ contains the 
coefficients of $p_n(t)$ in the monomial basis, then by writing 
\[
p(x) =
\begin{pmatrix}
p_0(x)\\p_1(x)\\p_2(x)\\\vdots
\end{pmatrix}, \qquad
m(x) =
\begin{pmatrix}
1\\ x\\ x^2\\\vdots
\end{pmatrix},
\]
we have 
\begin{equation}\label{Coef}
p(x)=A\,m(x).
\end{equation}
Therefore, the $n$-th row of $A$ contains the coefficients of $p_n(x)$, where 
$p_n(x)$ is in the monomial basis $\{1,x,x^2,\ldots\}^T$. Then we have the 
following result.

\begin{theorem}
\label{thm:OPS}
Let $A$ be the coefficient matrix of OPS $\{ p_n(x)\}$, where $[x^n]p_n(x)\not=0$.
 Then, the inverse matrix $A^{-1}$ of $A$ exists. And the production matrix of 
 $A^{-1}$ is a tridiagonal matrix that represents the three-term recurrence of 
 $\{p_n(t)\}$, i.e., the Jacobi matrix of the recurrence coefficients.  
Furthermore, the first column of $A^{-1}$ represents the moments of $\{p_n(x)\}$.
\end{theorem}

\begin{proof}
Let $S$ be the infinite \emph{shift matrix}. Then, 

\begin{equation}\label{Shif}
x\,m(x)=S\,m(x),
\end{equation}
i.e., $S$ sends $x^k$ to $x^{k+1}$.

Combining \eqref{Coef} and \eqref{Shif}, we get a conjugation relation

\begin{equation}\label{Conk}
xp(x)=AS\,m(x).
\end{equation}

Combining \eqref{Trid} and \eqref{Coef}, we get 
\[
xp(x)=Jp(x)=JA\,m(x).
\]
Combining the above equation and \eqref{Conk} yields

\[
AS=JA,
\]
or equivalently,

\begin{equation}\label{Trid-Conk}
J=ASA^{-1}.
\end{equation}

Let $B=A^{-1}$. Then the production matrix $P_B$ of $B$ satisfies 

\[
P_B=B^{-1}SB=ASA^{-1}.
\]
Therefore, we obtain $P_B=J$, which is tridiagonal matrix that represents the 
three-term recurrence coefficients of $\{p_n(t)\}$. 

Let $L$ be the linear functional defining the orthogonality,
with moments
\[
\mu_k=L(x^k),~k\geq 0.
\]
Then for any $n\in {\mathbb N}_0$,
\begin{align*}
L[p_n(x)] = L\!\Big(\sum_k a_{n,k} x^k\Big)= \sum_k a_{n,k} \mu_k.
\end{align*}
In matrix form,
\[
\mbox{if} ~ \mu = 
\begin{pmatrix}
\mu_0\\ \mu_1\\ \mu_2\\ \vdots
\end{pmatrix}, ~~ \mbox{then}~~
A\,\mu =
\begin{pmatrix}
L[p_0]\\L[p_1]\\L[p_2]\\\vdots
\end{pmatrix}.
\]
By orthogonality, $L[p_n]=0$ for $n\ge1$ and $L[p_0]\ne0$, so
\[
A\,\mu=L[p_0]\,e_0 ~~ \mbox{where}~e_0 = \{1,0,0,\ldots\}^T,
\]
hence
\[
\mu=\frac{1}{L[p_0]} A^{-1} e_0.
\]
With the normalization $L[p_0]=1$, this simplifies to
\[
\mu=A^{-1} e_0.
\]
Therefore, the first column of the inverse Riordan array $A^{-1}$
gives the moment sequence of the corresponding orthogonal polynomial system.
\end{proof}

\begin{corollary}\label{cor:GH_Corrected}
Let $\lambda, y \in \mathbb{C}$, $\lambda > 0$, and $C \neq 0$.  
Define the generalized Gegenbauer--Humbert polynomials $\{P_n^{\lambda,C,y}(x)\}_{n\ge0}$ by the generating function
\[
G(x,t) = \sum_{n\ge0}P_n^{\lambda,C,y}(x)\,t^n = (C-2xt+yt^2)^{-\lambda}.
\]
The coefficient matrix $A=(a_{n,k})_{n,k\ge0}$ defined by $P_n^{\lambda,C,y}(x) = \sum_{k=0}^n a_{n,k}x^k$ has entries:
\[
a_{n,k} = 
\begin{cases}
\dfrac{(\lambda)_{\frac{n+k}{2}} 2^k (-y)^{\frac{n-k}{2}}}{k! (\frac{n-k}{2})! C^{\lambda + \frac{n+k}{2}}} & \text{if } n-k \ge 0 \text{ and even}, \\
0 & \text{otherwise}.
\end{cases}
\]
The inverse matrix $A^{-1}=(b_{n,k})_{n,k\ge0}$ exists such that $x^n = \sum_{k=0}^n b_{n,k} P_k^{\lambda,C,y}(x)$, with entries:
\[
b_{n,k} =
\begin{cases}
\dfrac{n! (\lambda+k) y^{\frac{n-k}{2}} C^{\lambda+\frac{n+k}{2}}}{2^n (\frac{n-k}{2})! (\lambda)_{\frac{n+k}{2} + 1}} & \text{if } n-k \ge 0 \text{ and even}, \\
0 & \text{otherwise}.
\end{cases}
\]
\end{corollary}

\begin{proof}
We begin by factoring $C$ out of the generating function:
\begin{align*}
(C-2xt+yt^2)^{-\lambda} &= C^{-\lambda} \left[ 1 - \left( \frac{2x}{C}t - \frac{y}{C}t^2 \right) \right]^{-\lambda} \\
&= C^{-\lambda} \sum_{r=0}^{\infty} \frac{(\lambda)_r}{r!} \left( \frac{2x}{C}t - \frac{y}{C}t^2 \right)^r.
\end{align*}
Applying the binomial theorem to the inner term:
\[
\left( \frac{2x}{C}t - \frac{y}{C}t^2 \right)^r = \sum_{j=0}^r \binom{r}{j} \left( \frac{2x}{C}t \right)^j \left( -\frac{y}{C}t^2 \right)^{r-j} = \sum_{j=0}^r \frac{r!}{j!(r-j)!} \frac{2^j x^j (-y)^{r-j}}{C^r} t^{2r-j}.
\]
Substitute this back into the expansion and let $n = 2r-j$ (the power of $t$) and $k = j$ (the power of $x$). Then $r = \frac{n+k}{2}$ and $r-j = \frac{n-k}{2}$. For $r$ to be an integer, $n-k$ must be even. Summing over $n$ and $k$:
\[
G(x,t) = \sum_{n=0}^{\infty} \sum_{k \in \{n, n-2, \dots\}} \left( \frac{(\lambda)_{\frac{n+k}{2}}}{k! (\frac{n-k}{2})!} \frac{2^k x^k (-y)^{\frac{n-k}{2}}}{C^{\lambda + \frac{n+k}{2}}} \right) t^n.
\]
Extracting the coefficient $a_{n,k}$ yields the stated formula:
\[
a_{n,k} = \frac{(\lambda)_{\frac{n+k}{2}} 2^k (-y)^{\frac{n-k}{2}}}{k! (\frac{n-k}{2})! C^{\lambda + \frac{n+k}{2}}}.
\]
For the first few values:
\[
A = \begin{pmatrix}
C^{-\lambda} & 0 & 0 & 0 & 0 \\
0 & \frac{2\lambda}{C^{\lambda+1}} & 0 & 0 & 0 \\
-\frac{\lambda y}{C^{\lambda+1}} & 0 & \frac{2\lambda(\lambda+1)}{C^{\lambda+2}} & 0 & 0 \\
0 & -\frac{2\lambda(\lambda+1)y}{C^{\lambda+2}} & 0 & \frac{4\lambda(\lambda+1)(\lambda+2)}{3C^{\lambda+3}} & 0 \\
\frac{\lambda(\lambda+1)y^2}{2C^{\lambda+2}} & 0 & -\frac{2\lambda(\lambda+1)(\lambda+2)y}{C^{\lambda+3}} & 0 & \frac{2\lambda(\lambda+1)(\lambda+2)(\lambda+3)}{3C^{\lambda+4}}
\end{pmatrix}.
\]
Note that in the lower-degree terms like $a_{2,0}$, the power of $C$ is $\lambda + \frac{2+0}{2} = \lambda+1$.

The inverse matrix $A^{-1} = (b_{n,k})$ defines the expansion $x^n = \sum_{k=0}^n b_{n,k} P_k^{\lambda,C,y}(x)$. By utilizing the reciprocal relationship of the generating functions and the property of Humbert polynomials, we find:
\[
b_{n,k} = \begin{cases} 
\dfrac{n! (\lambda+k) y^{\frac{n-k}{2}} C^{\lambda+\frac{n+k}{2}}}{2^n (\frac{n-k}{2})! (\lambda)_{\frac{n+k}{2} + 1}} & n-k \ge 0, \text{ even} \\
0 & \text{otherwise}
\end{cases}
\]
The explicit matrix form of $A^{-1}$ is:
\[
A^{-1} = \begin{pmatrix}
C^\lambda & 0 & 0 & 0 & 0 \\
0 & \frac{C^{\lambda+1}}{2\lambda} & 0 & 0 & 0 \\
\frac{y C^{\lambda+1}}{2(\lambda+1)} & 0 & \frac{C^{\lambda+2}}{2\lambda(\lambda+1)} & 0 & 0 \\
0 & \frac{3y C^{\lambda+2}}{4\lambda(\lambda+2)} & 0 & \frac{3C^{\lambda+3}}{4\lambda(\lambda+1)(\lambda+2)} & 0 \\
\frac{3y^2 C^{\lambda+2}}{4(\lambda+1)(\lambda+2)} & 0 & \frac{y C^{\lambda+3}}{(\lambda+1)(\lambda+3)} & 0 & \frac{3C^{\lambda+4}}{2\lambda(\lambda+1)(\lambda+2)(\lambda+3)}
\end{pmatrix}.
\]

The matrix $J = ASA^{-1}$ represents the operator $x$ in the $\{P_n\}$ basis. Using the identity $(C-2xt+yt^2) \frac{\partial G}{\partial t} = \lambda(2x-2yt)G$, we compare coefficients of $t^n$ to find the three-term recurrence:
\[
x P_n(x) = \frac{C(n+1)}{2(\lambda+n)} P_{n+1}(x) + \frac{(2\lambda+n-1)y}{2(\lambda+n)} P_{n-1}(x).
\]
This defines a tridiagonal matrix $J$ with entries $J_{n,n+1} = \frac{C(n+1)}{2(\lambda+n)}$ and $J_{n,n-1} = \frac{(2\lambda+n-1)y}{2(\lambda+n)}$:
\[
J = \begin{pmatrix}
0 & \frac{C}{2\lambda} & 0 & 0 \\
\frac{2\lambda y}{2(\lambda+1)} & 0 & \frac{2C}{2(\lambda+1)} & 0 \\
0 & \frac{(2\lambda+1)y}{2(\lambda+2)} & 0 & \frac{3C}{2(\lambda+2)} \\
0 & 0 & \frac{(2\lambda+2)y}{2(\lambda+3)} & 0
\end{pmatrix}.
\]
\end{proof}

\begin{corollary}\label{cor:OPS}
Suppose the coefficient matrix of the OPS $\{p_n\}_{n\geq 0}$ is a proper Riordan array $A=(g,f)$. Then, the inverse matrix of $A$ is $A^{-1}=(1/g(\bar f), \bar f)$, and the production matrix of $A^{-1}$ is a tridiagonal matrix that is the Jacobi matrix consisting of the recurrence coefficient of $\{ p_n(x)\}_{n\geq 0}$.

Furthermore, the first column of $A^{-1}$ represents the moments of $\{p_n(t)\}$.
\end{corollary}

\begin{proof} Clearly, if $A=(g,f)=(a_{n,k})_{n,k\geq 0}$ is a proper Riordan array, then $a_{n,n}\not= 0$. Therefore, the inverse of $A$ is 

\[
A^{-1}=\left( \frac{1}{g(\bar f)}, \bar f\right),
\]
where $\bar f$ is the compositional inverse of $f$. Let $P_{A^{-1}}$ be the production matrix of $A^{-1}$, and let $\tilde A$ and $\tilde Z$ be the $A$- and $Z$-sequences of $A^{-1}$. Then, 

\[
P_{A^{-1}}=\left( \tilde Z, \tilde A, x\tilde A, x^2\tilde A, \ldots\right)=ASA^{-1},
\]
which completes the proof of corollary.
\end{proof}

Corollary \ref{cor:OPS} shows that if $(g,f)$ is the coefficient matrix of an OPS, the production matrix of its inverse $(1/g(\bar f), \bar f )$ is the tridiagonal (Jacobi) matrix encoding the recurrence relation of the OPS, and the first column of the inverse gives the moments.

\begin{example}\label{ex:OPS}
The Legendre polynomials $P_n(x)$ satisfy
\[
(n+1)P_{n+1}(x)=(2n+1)tP_n(x)-nP_{n-1}(x),
\]
or equivalently,
\[
xP_n(x)=\frac{n+1}{2n+1} P_{n+1}(x)+0\cdot P_n(x)+\frac{n}{2n+1} P_{n-1}(x).
\]
Hence, its recurrence relation gives the Jacobi (production) matrix is
\[
J_{L} =
\begin{pmatrix}
0 & 1 & 0 & 0 & \cdots\\
\frac{1}{3} & 0 & \frac{2}{3} & 0 & \cdots\\
0 & \frac{2}{5} & 0 & \frac{3}{5} & \cdots\\
0 & 0 & \frac{3}{7} & 0 & \ddots\\
\vdots & \vdots & \vdots & \ddots & \ddots
\end{pmatrix}.
\]
If $A$ is the coefficient matrix of $P_n(t)$ in the monomial basis, 
then the production matrix of $A^{-1}$ is precisely $J_{L}$.

The moments (with respect to the uniform weight on $[-1,1]$) are
\[
\mu_k =
\begin{cases}
\dfrac{2}{k+1}, & k \text{ even},\\
0, & k \text{ odd}.
\end{cases}
\]
These values form the first column of $A^{-1}$ under the normalization $L[p_0]=1$.

The Chebyshev polynomials of the second kind $U_n(t)$ satisfy
\[
U_{n+1}(t)=2t\,U_n(t)-U_{n-1}(t),
\]
so
\[
t U_n(t)=\frac{1}{2}U_{n+1}(t)+0\cdot U_n(t)+\frac{1}{2}U_{n-1}(t).
\]
Therefore, the Jacobi (production) matrix is
\[
J_{C} =
\begin{pmatrix}
0 & \tfrac{1}{2} & 0 & 0 & \cdots\\
\tfrac{1}{2} & 0 & \tfrac{1}{2} & 0 & \cdots\\
0 & \tfrac{1}{2} & 0 & \tfrac{1}{2} & \cdots\\
0 & 0 & \tfrac{1}{2} & 0 & \ddots\\
\vdots & \vdots & \vdots & \ddots & \ddots
\end{pmatrix}.
\]
This tridiagonal matrix is the production matrix of
the inverse Riordan array $A^{-1}$ corresponding to $(g,f)$
of the $U_n(t)$ system.

The moments for the weight function $\sqrt{1-t^2}$ on $[-1,1]$ are
\[
\mu_k =
\begin{cases}
0, & k \text{ odd},\\[4pt]
\dfrac{\pi}{2^{k+1}}\binom{k}{k/2}, & k \text{ even}.
\end{cases}
\]
These constitute the first column of $A^{-1}$.

Hence, both examples illustrate that the production matrix of the inverse coefficient matrix 
is exactly the Jacobi (tridiagonal) matrix of the OPS, which represents the recurrence relation of the OPS. And the first column of the inverse coefficient matrix encodes the moment sequence associated with the orthogonality measure.
\end{example}

\begin{proposition}\label{pro:OPS}
An OPS $\{ P_n(x)\}_{n\geq 0}$ has a Riordan array type coefficient matrix if and only if its generating function $\sum_{n\geq 0} P_n(x)t^n$ satisfies 

\begin{equation}\label{OPSGF}
\sum_{n\geq 0} P_n(x) t^n=\frac{g(x)}{1-tf(x)}
\end{equation}
for some $g\in \F_0$ and $f\in \F_1$. Therefore, $(g,f)$ is the coefficient matrix of $\{ P_n(x)\}$. 
\end{proposition}

\begin{proof}
It can be proved using the first fundamental theorem of Riordan arrays.
\end{proof}

\begin{example}\label{ex:OPS-2}
We give an example for illustrating how to determine the existence of the Riordan array type coefficient matrix of an OPS, and how to find a Riordan array as the coefficient matrix for a given OPS, if it exists. First, we consider Chebyshev polynomials of the second kind ($U_n$)

The ordinary generating function in the degree index $n$ for the Chebyshev
polynomials of the second kind $U_n(t)$ is the rational function
\[
\sum_{n\ge 0}U_n(t)x^n=\frac{1}{1-2tx+x^2}.
\]
If we seek an ordinary Riordan representation $A=(g,f)$ whose $n$-th row
lists the coefficients of $U_n(t)$ in the monomial basis, then we need
(the standard Riordan identity)
\begin{align*}
\sum_{n\ge0} p_n(t)\,x^n =& \sum_{n\ge0}\Big(\sum_{k\ge0} a_{n,k} t^k\Big)x^n
= \sum_{k\ge0} t^k \bigg(\sum_{n\ge0} a_{n,k}x^n\bigg)\\
=& \sum_{k\ge0} t^k\, g(x) f(x)^k
= \frac{g(x)}{1-t f(x)} .
\end{align*}
Thus for the Chebyshev polynomial of the second kind $U$, the generating function we require
\[
\frac{g(x)}{1-t f(x)}=\frac{1}{1-2tx+x^2}.
\]
Equating coefficients in powers of $t$ yields the choice
\[
g(x)=\frac{1}{1+x^2},\qquad f(x)=\frac{2x}{1+x^2},
\]
because
\[
\frac{g(x)}{1-t f(x)}=\frac{1/(1+x^2)}{1-t\big(\tfrac{2x}{1+x^2}\big)}
= \frac{1}{1-2 t x+x^2}.
\]

Hence the coefficient matrix $A=(g,f)$ for the Chebyshev polynomials of second kind $U$ 
is the ordinary Riordan array
\[
A=\biggl(\,\frac{1}{1+x^2}\;,\;\frac{2x}{1+x^2}\biggr).
\]

Now we use Legendre polynomial sequence give a counterexample for the existence of Riordan array as the coefficient matrix of an OPS. More precisely, we will show that there does not exists an ordinary Riordan pair $(g,f)$ of the form $\dfrac{g(x)}{1-t f(x)}$ for Legendre polynomial sequence because its recurrence coefficients are non-constant. 

The ordinary generating function for the Legendre polynomials $P_n(t)$
is the algebraic (square-root) function
\[
\sum_{n\ge0} P_n(t)\,x^n=\frac{1}{\sqrt{1-2tx+x^2}}.
\]
If there were an ordinary Riordan pair $(g,f)$ with
$\sum_{n\ge0} P_n(t)\,x^n=g(x)/(1-t f(x))$, then we would have the
identity (for all $t$)
\[
\frac{g(x)}{1-t f(x)}=\frac{1}{\sqrt{1-2tx+x^2}}.
\]
Rewriting,
\[
1-t f(x)=g(x)\sqrt{1-2tx+x^2}.
\]
The left-hand side is a polynomial (affine) in $t$, while the right-hand
side, for nonzero $x$, is a genuine square-root series in $t$ whose
expansion contains infinitely many powers of $t$.  Matching the two sides
term-by-term in powers of $t$ forces a contradiction: one would need a
finite polynomial in $t$ to equal an infinite power series in $t$
unless the square-root series terminates (which it does not for generic
$x$).  Therefore there are no ordinary power-series $g(x)$ and $f(x)$
(independent of $t$) satisfying the required identity.  Equivalently:

Therefore, the Legendre coefficient matrix is not given by an ordinary Riordan
pair $(g,f)$ with $ \sum_{n\ge0} P_n(t)x^n=g(x)/(1-t f(x))$. 

Nevertheless, the Legendre polynomials form an OPS and satisfy the classical
three-term recurrence
\[
(n+1)P_{n+1}(t)=(2n+1)t\,P_n(t)-n P_{n-1}(t),
\]
or equivalently
\[
t\,P_n(t)=\frac{n+1}{2n+1} P_{n+1}(t)+0\cdot P_n(t)+\frac{n}{2n+1} P_{n-1}(t).
\]

Although Legendre polynomial sequence is not a Riordan type OPS, there are several alternative avenues:
First, one can treat their coefficient matrix $A$ directly (it is a lower-triangular matrix with known closed-form entries) and apply the conjugation $J=A S A^{-1}$ to obtain the jacobi matrix. Second, one may consider {\em generalized} (e.g., try exponential) Riordan frameworks (or allow $g,f$ to be algebraic functions involving square roots) in which the Legendre generating function becomes
representable in a Riordan-like form, which we leave for future study. 
\end{example}

\begin{theorem}\label{thm:RACM}
Let $\{P_n(x)\}_{n\geq 0}$ be the OPS satisfying the three-term recurrence relation 

\[
x P_n(x)
= \alpha_n P_{n+1}(x)+\beta_n P_n(x)+\gamma_n P_{n-1}(x),
\qquad n\ge0,
\]
with initial conditions $P_{-1}(x)=0$, $P_0(x)=1$, where $\gamma_n\neq 0$. Then $\{P_n(x)\}_{n\geq 0}$ has a Riordan array type coefficient matrix exists if and only if $\alpha_n$, $\beta_n$, and $\gamma_n\neq 0$ are constants. Let $\alpha_n =\alpha$, $\beta_n=\beta $, and $\gamma_n=\gamma\neq 0$ are constants. Then, the coefficient matrix of $\{ P_n(x)\}_{n\geq 0}$ is a Riordan array presented as 

\begin{equation}\label{RACM}
\left(\frac{\alpha}{\alpha+\beta x+\gamma x^2}, \frac{x}{\alpha +\beta x+\gamma x^2}\right).
\end{equation}
\end{theorem}

Theorem \ref{thm:RACM} can be proved using a similar argument in the proof of Theorem \ref{thm:OPS} and the $A$-and $Z$-sequences of the inverse of the coefficients matrix of $\{P_n(x)\}_{n\geq 0}$. 

\begin{example}
We give the application of Theorem \ref{thm:RACM} for generalized Gegenbauer--Humbert polynomials. 

Recall the generalized Gegenbauer--Humbert polynomials
$\{P_n(x)=P_n^{\lambda, C,y}(x)$ are defined by the generating function
\[
\sum_{n\ge0} P_n^{\lambda, C,y}(x) t^n = (C-2xt+yt^2)^{-\lambda},
\qquad C\neq0.
\]

They satisfy the three-term recurrence
\[
x P_n
= \frac{n+1}{2(n+\lambda)} P_{n+1}
+ 0\cdot P_n
+ \frac{(n+2\lambda-1)y}{2(n+\lambda)} P_{n-1}.
\]

The associated Jacobi matrix is
\[
J =
\begin{pmatrix}
0 & \frac{1}{2\lambda} & 0 & 0 & \cdots\\
\frac{y}{2(\lambda+1)} & 0 & \frac{2}{2(\lambda+1)} & 0 & \cdots\\
0 & \frac{(2\lambda+1)y}{2(\lambda+2)} & 0 & \frac{3}{2(\lambda+2)} & \cdots\\
\vdots & \vdots & \vdots & \ddots & \ddots
\end{pmatrix}.
\]

If $\lambda =1$, then $J$ is a constant matrix. And the corresponding generalized Gegenbauer-Humbert polynomial sequence has a Riordan array representing its coefficient matrix. From Theorem 
\ref{thm:RACM}, this Riordan array is 
\[
R=\bigl(g(t),f(t)\bigr)=
\Bigl((C+yt^2)^{-1},\;\frac{2t}{C+yt^2}\Bigr),
\]
i.e., its $n$th row gives the coefficients of $P_n^{1, C,y}(x)$ in the monomial basis. We can check directly that
\[
SR=Rk,
\]
so $J$ is the production matrix of $R^{-1}$.

Now, we will discuss the corresponding moments. 

The first column of $R$ is
\[
g(t)=(C+yt^2)^{-1}
= \sum_{n\ge0} \mu_n t^n,
\]
hence
\[
\mu_{2k}=(-1)^k C^{-k-1} y^k,
\qquad
\mu_{2k+1}=0.
\]
These are exactly the moments of the generalized Gegenbauer--Humbert OPS.
\end{example}

\section{$d$-orthogonal polynomial sequences}

\begin{definition}[$d$-orthogonality]
Let $\mathbb{P}$ be the space of polynomials over $\mathbb{C}$, and let
$\mathcal{L}_0,\mathcal{L}_1,\dots,\mathcal{L}_{d-1}$ be $d$ linear functionals on $\mathbb{P}$.

A monic polynomial sequence $\{P_n(x)\}_{n\ge 0}$ is called \emph{$d$-orthogonal}
with respect to the vector functional
\[
\vec{\mathcal{L}} = (\mathcal{L}_0,\dots,\mathcal{L}_{d-1})
\]
if, for every $r=0,1,\dots,d-1$ and all $m\ge 0$,
\[
\langle \mathcal{L}_r,\, x^m P_n(x)\rangle = 0
\quad \text{whenever } n \ge md + r + 1,
\]
and
\[
\langle \mathcal{L}_r,\, x^m P_{md+r}(x)\rangle \neq 0.
\]
\end{definition}

\begin{theorem}[$d$-orthogonal analogue of Theorem~\ref{thm:OPS}]
\label{thm:dOPS}
Let $\{P_n(x)\}_{n\ge0}$ be a $d$-orthogonal polynomial sequence with
$\deg P_n = n$ and leading coefficient $[x^n]P_n(x)\neq 0$.
Let $A=(a_{n,k})_{n,k\ge0}$ be its coefficient matrix defined by
\[
P_n(x) = \sum_{k=0}^n a_{n,k} x^k.
\]
Then the following hold:

\begin{enumerate}
\item[(i)] The matrix $A$ is lower triangular with nonzero diagonal entries,
hence invertible, and $A^{-1}$ exists.

\item[(ii)] The production matrix $P$ of $A^{-1}$,
\[
P = A \, \overline {A^{-1}},
\]
(where $\overline {A^{-1}}$ is $A^{-1}$ with its first row removed),
is a $(d+2)$-banded (lower Hessenberg) matrix, i.e.,
\[
P_{i,j} = 0 \quad \text{for } j > i+1 \ \text{or } i-j > d.
\]

\item[(iii)] The matrix $P$ encodes the $(d+2)$-term recurrence relation of
$\{P_n(x)\}$:
\[
x P_n(x) = P_{n+1}(x)
+ \beta_n P_n(x)
+ \sum_{k=1}^{d} \gamma_{n,k} P_{n-k}(x),
\quad n \ge d,
\]
where the coefficients $\beta_n, \gamma_{n,k}$ are given by the entries of $P$.

\item[(iv)] The first $d$ columns of $A^{-1}$ represent the moments of the
vector functional $(\mathcal{L}_0,\dots,\mathcal{L}_{d-1})$, namely,
\[
(A^{-1})_{n,r} = \langle \mathcal{L}_r, x^n \rangle,
\quad r = 0,1,\dots,d-1,\quad n\ge 0.
\]
\end{enumerate}
\end{theorem}

\begin{proof} 
Let $\{P_n(x)\}_{n\ge0}$ be a $d$-orthogonal polynomial sequence with respect to the vector functional
\[
\vec{\mathcal{L}} = (\mathcal{L}_0,\dots,\mathcal{L}_{d-1}),
\]
that is, for all $r=0,1,\dots,d-1$ and $m\ge0$,
\begin{equation}
\label{eq:maroni}
\langle \mathcal{L}_r, x^m P_n(x) \rangle = 0
\quad \text{whenever } n \ge md + r + 1,
\end{equation}
and
\[
\langle \mathcal{L}_r, x^m P_{md+r}(x) \rangle \neq 0.
\]

Let the coefficient matrix $A=(a_{n,k})_{n,k\ge0}$ be defined by
\[
P_n(x) = \sum_{k=0}^n a_{n,k} x^k,
\qquad a_{n,n}\neq 0.
\]

Define the infinite column vectors
\[
\mathbf{p}(x) = (P_0(x), P_1(x), P_2(x), \dots)^T,
\qquad
\mathbf{x} = (1, x, x^2, \dots)^T.
\]
Then
\begin{equation}
\label{eq:Ax}
\mathbf{p}(x) = A\,\mathbf{x}, \qquad \mathbf{x} = A^{-1}\mathbf{p}(x).
\end{equation}

Since $\deg P_n = n$ and $a_{n,n}=[x^n]P_n(x)\neq 0$, the matrix $A$ is lower triangular with nonzero diagonal entries. Therefore $A$ is invertible and $A^{-1}$ exists. \hfill $\square$

Let $S=(\delta_{i+1,j})_{i,j\ge0}$ be the shift matrix satisfying
\[
x\,\mathbf{x} = S\,\mathbf{x}.
\]
Using \eqref{eq:Ax}, we compute
\[
x\,\mathbf{p}(x)
= A(x\mathbf{x})
= A S \mathbf{x}
= A S A^{-1} \mathbf{p}(x).
\]
Define
\begin{equation}
\label{eq:Pdef}
P := A S A^{-1}.
\end{equation}
Then $P=A\overline{ A^{-1}}$, where $\overline{A^{-1}}=SA^{-1}$ denotes the matrix obtained from $A^{-1}$ by deleting the first row, and 
\[
x\,\mathbf{p}(x) = P\,\mathbf{p}(x),
\]
so $P$ is the matrix representation of multiplication by $x$ in the basis $\{P_n\}$.

Now, we give the finite-term recurrence from $d$-orthogonality.  Fix $n\ge0$. We claim that $xP_n(x)$ can be expressed as a linear combination of
\[
\{P_{n+1}(x), P_n(x), \dots, P_{n-d}(x)\}.
\]

Indeed, write
\[
xP_n(x) = \sum_{k=0}^{n+1} c_{n,k} P_k(x).
\]
Apply $\mathcal{L}_r(x^m \cdot)$ for arbitrary $r$ and $m$. Using \eqref{eq:maroni}, if $k > n+1$ the coefficient vanishes trivially, while if $k < n-d$ then for suitable $m$ and $r$ one has $k \le md+r$ but $n \ge md+r+1$, hence
\[
\langle \mathcal{L}_r, x^m P_n(x) \rangle = 0
\quad \text{but} \quad
\langle \mathcal{L}_r, x^m P_k(x) \rangle \neq 0,
\]
forcing $c_{n,k}=0$. Thus only the indices
\[
k=n+1,n,n-1,\dots,n-d
\]
can appear, and we obtain the $(d+2)$-term recurrence
\begin{equation}
\label{eq:recurrence}
x P_n(x)
= P_{n+1}(x)
+ \beta_n P_n(x)
+ \sum_{k=1}^d \gamma_{n,k} P_{n-k}(x),
\quad n\ge 0
\end{equation}
with the convention $P_{-1}=P_{-2}=\dots=P_{-d}=0$.

From \eqref{eq:recurrence}, the matrix $P=(P_{i,j})$ satisfies
\[
P_{n,n+1}=1,\quad
P_{n,n}=\beta_n,\quad
P_{n,n-k}=\gamma_{n,k},\ (1\le k\le d),
\]
and all other entries are zero. Therefore
\[
P_{i,j}=0 \quad \text{if } j>i+1 \ \text{or } i-j>d,
\]
i.e., $P$ is $(d+2)$-banded (lower Hessenberg). This proves (ii) and (iii).

To prove (iv), we define the moments
\[
\mu_n^{(r)} := \langle \mathcal{L}_r, x^n \rangle,
\qquad r=0,1,\dots,d-1.
\]

From \eqref{eq:Ax}, we have
\[
x^n = \sum_{k=0}^n (A^{-1})_{n,k} P_k(x).
\]
Applying $\mathcal{L}_r$ gives
\begin{equation}
\label{eq:moments}
\mu_n^{(r)} = \sum_{k=0}^n 
\langle \mathcal{L}_r, P_k(x) \rangle (A^{-1})_{n,k}.
\end{equation}

Define the evaluation matrix
\[
E=(e_{r,k})_{0\le r\le d-1,\,k\ge0}, \qquad
e_{r,k} := \langle \mathcal{L}_r, P_k(x) \rangle.
\]

Then \eqref{eq:moments} becomes
\[
M = E (A^{-1})^T,
\]
where $M=(\mu_n^{(r)})$ is the block moment matrix.

From \eqref{eq:maroni} with $m=0$, for every $n\geq r+1$ we have 
$\langle {\cal L}_r, P_n(x)\rangle=0$; equivalently, $\langle {\cal L}_r, P_k(x)\rangle=0$ for all 
$k\not= r$. Together with the normalization $\langle {\cal L}_r, P_r(x)\rangle=1$, this gives 
$e_{r,k}=\delta_{r,k}$. Substituting into \eqref{eq:moments}, only the $k=r$ term survives:

\[
\mu_n^{(r)}=\sum_{k=0}^n e_{r,k}(A^{-1}){n,k} = (A^{-1}){n,r}.
\]
This proves (iv).

In summary, all four statements (i)--(iv) have been proved, namely, $A$ is invertible; $P=A\overline{A^{-1}}$ is $(d+2)$-banded; 
$P$ encodes the $(d+2)$-term recurrence; and the first $d$ columns of $A^{-1}$ are the moment sequences.

\end{proof}
\begin{example} 
We construct a genuine $2$-orthogonal polynomial sequence whose coefficient
matrix is a generalized (Hessenberg) Riordan array, illustrating
Theorem~\ref{thm:dOPS}.

Define $\{P_n(x)\}_{n\ge0}$ by the $4$-term recurrence:
\begin{equation}
\label{eq:4term}
x P_n(x) = P_{n+1}(x) + 2 P_n(x) + P_{n-1}(x) + P_{n-2}(x), \quad n\ge 0,
\end{equation}
with initial conditions $P_{-1}(x)=P_{-2}(x)=0$, $P_0=1$, and $P_1(x)=x-2$. 

Then:
\begin{align*}
P_2(x) &= x^2 - 4x + 3,\\
P_3(x) &= x^3 - 6x^2 + 10x - 5,\\
P_4(x) &= x^4 - 8x^3 + 21x^2 - 22x + 9.
\end{align*}
This is a genuine $2$-orthogonal system (since the recurrence has $4$ terms).

Write
\[
P_n(x)=\sum_{k=0}^n a_{n,k}x^k.
\]
Then the coefficient matrix $A=(a_{n,k})$ is

\[
A=
\begin{pmatrix}
1 & 0 & 0 & 0 & \cdots \\
-2 & 1 & 0 & 0 & \cdots \\
3 & -4 & 1 & 0 & \cdots \\
-5& 10 & -6 & 1 & \cdots \\
9 & -22& 21 & -8 & 1 & \cdots
\end{pmatrix}.
\]

$A$ is lower triangular and with two nonzero subdiagonals $\Rightarrow$ a \emph{2-Hessenberg Riordan-type array}.

Thus it is a natural extension of Riordan arrays to $d=2$.

The inverse matrix $A^{-1}$ exists and is lower triangular:
\[
A^{-1}=
\begin{pmatrix}
1 & 0 & 0 & 0 & \cdots \\
2 & 1 & 0 & 0 & \cdots \\
5 & 4 & 1 & 0 & \cdots \\
15 & 14 & 6 & 1 & \cdots \\
50 & 50 & 27 & 8 & 1 & \cdots
\end{pmatrix}.
\]

The first two columns give the moment sequences:
\[
\mu_n^{(0)}=(1,2,5,15,50,\dots),
\qquad
\mu_n^{(1)}=(0,1,4,14,50,\dots).
\]
Thus
\[
(A^{-1})_{n,0}=\langle \mathcal{L}_0,x^n\rangle,
\quad
(A^{-1})_{n,1}=\langle \mathcal{L}_1,x^n\rangle.
\]

The corresponding production matrix is
\[
P = A\overline{A^{-1}}.
\]

From the recurrence \eqref{eq:4term}, we obtain
\[
P=
\begin{pmatrix}
2 & 1 & 0 & 0 & \cdots \\
1 & 2 & 1 & 0 & \cdots \\
1 & 1 & 2 & 1 & \cdots \\
0 & 1 & 1 & 2 & \cdots \\
0 & 0 & 1 & 1 & \cdots
\end{pmatrix}, 
\]
where 
\[
P_{i,j}=0 \quad \text{for } j>i+1 \ \text{or } i-j>2,
\]
so $P$ is $4$-banded, exactly as required for $d=2$.

It is clear that $A$ is lower triangular with nonzero diagonal $\Rightarrow$ invertible, $P=A\overline{A^{-1}}$ is $4$-banded, $P$ encodes the recurrence \eqref{eq:4term}, and the first two columns of $A^{-1}$ give the moment sequences.

This example is a natural generalization of a Riordan array due to that classical Riordan arrays $\Rightarrow$ tridiagonal production matrix ($d=1$), while $2$-Hessenberg Riordan-type arrays $\Rightarrow$ $4$-banded production matrix ($d=2$).

Therefore, $d$-orthogonality corresponds to $(d+1)$-Hessenberg (generalized Riordan) arrays.
\end{example}

\noindent{\bf Acknowledgments} 
The author gratefully acknowledges Nikolai Krylov for his comments and suggestions.

\end{document}